\documentstyle[12pt, amssymb,amsthm]{article}

\DeclareSymbolFont{rsfs}{U}{rsfs}{m}{n}
\DeclareSymbolFontAlphabet{\mathscr}{rsfs}

\def\dd{{\mathfrak d}}

\def\gg{{\mathfrak g}} 
 
\def\ii{{\mathfrak i}}

\def\rr{{\mathfrak r}}

\def\sl{{\mathfrak sl}}

\def\VV{{\mathfrak V}}
\def\WW{{\mathfrak W}}

\def\Vir{{\VV \ii \rr}}

\def\CCC{{\mathbb C}}

\def\NNN{{\mathbb N}}

\def\RRR{{\mathbb R}} 
\def\SSS{{\mathbb S}}

\def\ZZZ{{\mathbb Z}} 

\def\A{{\mathcal A}}

\def\C{{\mathcal C}}

\def\F{{\mathcal F}}

\def\L{{\mathcal L}}

\def\S{{\mathcal S}}

\def\U{{\mathcal U}}

\def\t1{{\frac{1}{3}}}
\def\1/2{{\frac{1}{2}}}
\def\3/2{{\frac{3}{2}}}
\def\CCCC{{\mathscr C}} 

\def\Ž{{\' e}}
\def\{{\` e}}
\def\{{\^ e}}
\def\ˆ{{\` a}}

 \newtheorem{lemma}{Lemma}[section]
 \newtheorem{proposition}[lemma]{Proposition}
  \newtheorem{examples}[lemma]{Examples}
  \newtheorem{example}[lemma]{Example} 
  \newtheorem{theorem}[lemma]{Theorem}
\newtheorem{definition}[lemma]{Definition}
\newtheorem{notation}[lemma]{Notation}

 \newtheorem{pcorollary}[lemma]{Proof's corollary}
 \newtheorem{corollary}[lemma]{Corollary}
 \newtheorem{remark}[lemma]{Remark}
 \newtheorem{formula}[lemma]{Formula}

  \newtheorem{summary}[lemma]{Summary}

\title{Neveu-Schwarz and operators algebras I    \\ \textit{Vertex operators superalgebras}   }

\author{S\'ebastien Palcoux }

 \date{}

\begin{document}
 \maketitle
\begin{abstract}
This paper is the first of a series giving a self-contained way from the Neveu-Schwarz algebra to a new series of irreducible subfactors.
Here we present an elementary, progressive and self-contained approch to vertex operator superalgebra. We then build such a structure from the loop algebra $L \gg$ of any simple finite dimensional Lie algebra $\gg$. The Neveu-Schwarz algebra $\Vir_{1/2}$ emerges naturally on. 
As application, we obtain unitary actions of $\Vir_{1/2}$ on the unitary  series of $L \gg$.

\end{abstract}

\tableofcontents

\section{Introduction}

\subsection{Background of the series}
In the $90$'s,  V. Jones and A. Wassermann started a program whose goal is to understand the unitary conformal field theory from the point of view of operator algebras (see \cite{2c}, \cite{2b}).  In \cite{2}, Wassermann defines and computes the Connes fusion of the irreducible positive energy representations of the loop group $LSU(n)$ at fixed level $\ell$, using primary fields, and with consequences in the theory of subfactors.  In  \cite{vtl} V. Toledano Laredo proves the Connes fusion rules for $LSpin(2n)$ using similar methods. Now, let Diff$(\SSS^{1})$ be the diffeomorphism group on the circle, its Lie algebra is the Witt algebra $\WW$ generated by $d_{n}$ ($n \in \ZZZ$), with $[d_{m} , d_{n}] = (m-n)d_{m+n}$. It admits a unique central extension called the Virasoro algebra $\Vir$. Its unitary positive energy representation theory and the character formulas can be deduced by a so-called Goddard-Kent-Olive (GKO) coset construction from the theory of $LSU(2)$ and the Kac-Weyl formulas (see \cite{1}, \cite{3a}). In \cite{loke}, T. Loke uses the coset construction to compute the Connes fusion for $\Vir$. 
Now, the Witt algebra admits two supersymmetric extensions $\WW_{0}$ and $\WW_{1/2}$ with central extensions called the Ramond and the Neveu-Schwarz algebras, noted $\Vir_{0}$ and $\Vir_{1/2}$. 

In this series (this paper, \cite{NSOAII} and  \cite{NSOAIII} ), we naturally introduce $\Vir_{1/2}$ in the vertex superalgebra context of $L\sl_{2}$, we give a complete proof of the classification of its unitary positive energy representations, we obtain directly their character; then we give the Connes fusion rules, and an irreducible finite depth type II$_{1}$ subfactors for each representation of the discrete series.
Note that we could do the same for the Ramond algebra $\Vir_{0} $, using twisted vertex module over the vertex operator algebra of the Neveu-Schwarz algebra $\Vir_{1/2}$, as  R. W. Verrill  \cite{verrill} and  Wassermann \cite{wass2} do for twisted loop groups.

\subsection{Overview of the paper}
$\begin{array}{c} \end{array}$ \hspace{0,25cm}  First, we look unitary, projective, positive energy representations of $\WW_{1/2}$.  \\ The projectivity gives  $2$-cocycles, so that $\WW_{1/2}$ admits a unique central extension $\Vir_{1/2}$. Such representations are completely reducible, and the irreducibles are given by the unitary highest weight representations of $\Vir_{1/2}$: Verma modules $V(c,h)$ quotiented by null vectors, in no-ghost cases. 

 From the  fermion algebra on $H= \F_{NS}$, we build the fermion field $\psi(z)$. Locality and Dong's lemma permit, via OPE,  to generate a set of fields $\S$, so that there is a $1-1$ map $V:   H   \rightarrow    \S  $, with $Id = V(\Omega)$ and a  Virasoro field $L = V(\omega)$. Then, we give vertex operator superalgebra's axioms, permitting to come so far, in a general framework $( H, V, \Omega, \omega) $, with $H$ prehilbert. 
 
Let $\gg$ a simple  finite-dimensional Lie algebra, $\widehat{\gg}_{+}$ the $\gg$-boson algebra \\Ê (central extension of the loop algebra $L\gg$) and $\widehat{\gg}_{-}$ the $\gg$-fermion algebra. \\Ê  We build a module vertex operator superalgebra from $\widehat{\gg} = \widehat{\gg}_{+} \ltimes \widehat{\gg}_{-} $ on $H = L(V_{\lambda}, \ell) \otimes \F_{NS}^{\gg}$, so that $\Vir_{1/2}$ acts on  with ÊÊ  $(c,h) =(\frac{3}{2}  \cdot \frac{\ell +1/3 g}{\ell + g} dim(\gg), \frac{c_{V_{\lambda}}}{2(\ell + g)})$, with $g$  the dual Coxeter number and $c_{V_{\lambda}}$ the Casimir number. 

\subsection{The Neveu-Schwarz algebra}  
We start with $\WW_{1/2}$, the Witt superalgebra of sector (NS):  \begin{center}
$  \left\{  \begin{array}{l} 
\lbrack d_{m},d_{n} \rbrack \hspace{0,25cm}  = (m-n)d_{m+n}   \hspace{0,45cm}  m, n \in \ZZZ \\ 
\lbrack \gamma_{m},d_{n} \rbrack \hspace{0,25cm}   = (m-\frac{n}{2})\gamma_{m+n}  \hspace{0,45cm}  m \in \ZZZ + \1/2,  n \in \ZZZ   \\ 
\lbrack \gamma_{m},\gamma_{n} \rbrack_{+} = 2d_{m+n}  \hspace{1,72cm} m, n \in \ZZZ + \1/2
\end{array}   \right.  $ \end{center}
together with $d^{\star}_{n} =d_{-n}$ and  $\gamma^{\star}_{m} = \gamma_{-m}$;  we study representations which are: \\
 \textbf{(a)} Unitary:  $\pi (A)^{\star} = \pi (A^{\star})$ \\
 \textbf{(b)} Projective:  $A \mapsto \pi (A)$ is linear and $[\pi (A) , \pi (B)] - \pi ([A,B]) \in \CCC$. \\
 \textbf{(c)} Positive energy : $H$ admits an orthogonal decomposition $ H = \bigoplus_{n \in  \frac{1}{2}\NNN} H_{n}$ \\
 such that $\exists D$  acting on $H_{n}$ as multiplication by $n$,   $H_{0} \ne \{ 0 \} $,   dim$(H_{n}) < + \infty $\\   Here, $ \exists h \in \CCC $  such that  $D   = \pi (d_{0}) -  hI $. \\  
Now, the projectivity gives the 2-cocycles and we see that $H_{2}(\WW_{1/2}, \CCC)$ is $1$-dimensional, $\WW_{1/2}$ admits a unique central extension up to equivalence:
\begin{center} $0  \rightarrow  H_{2}(\WW_{1/2}, \CCC)  \rightarrow   \Vir_{1/2}    \rightarrow  \WW_{1/2}     \rightarrow  0  $  \end{center}
$\Vir_{1/2}$ is the SuperVirasoro (of sector NS) or Neveu-Schwarz algebra: 
\begin{center}  $ \left\{  \begin{array}{l} 
\lbrack L_{m},L_{n} \rbrack \hspace{0,3cm} = (m-n)L_{m+n} +\frac{C}{12}(m^{3} - m) \delta_{m+n}  Ê \\
\lbrack G_{m},L_{n} \rbrack \hspace{0,28cm} = (m-\frac{n}{2})G_{m+n}    \\
\lbrack G_{m},G_{n}\rbrack_{+} = 2L_{m+n} + \frac{C}{3}(m^{2}- \frac{1}{4})\delta_{m+n} 
\end{array}   \right.  $ \end{center}
 with   $L_{n}^{\star} = L_{-n} $,  $G_{m}^{\star} = G_{-m} $ and  $C=cI$,  $c \in \CCC$ called the \textbf{central charge}. \\ 
The representations are completely reducible,  the irreducibles are determined by the two numbers $c, h$,  and are completely given by unitary highest weight representations of 
$\Vir_{1/2}  $, described as follows: 
The Verma modules  $H= V(c,h)$ are freely generated by:  $0 \neq \Omega \in H$ (cyclic vector),  $C\Omega = c\Omega$ , $L_{0}\Omega = h\Omega$ and $\Vir^{+}_{1/2}\Omega = \{0 \}$. Now, $(\Omega , \Omega )= 1$, $\pi (A)^{\star} = \pi(A^{\star})$ and $(u,v) = \overline{(v,u)} $ give  the sesquilinear form $( . , . )$.  $V(c,h)$ can admit ghost: $(u,u) < 0$, and null vectors:  $(u,u) = 0$ .   In no ghost case, the set of null vectors is $ K(c,h)$ the kernel of $( . , . )$, the maximal proper submodule. \\  Let $L(c,h)= V(c,h)/K(c,h)$, the unitary highest weight representations. \\ Theorem $1.2$ of \cite{NSOAII}  will be proved classifying no ghost cases. 

\subsection{Vertex operators superalgebras} 
Our approch of vertex operators superalgebras is freely inspired by the followings references: Borcherds \cite{3d}, Goddard \cite{3c}, Kac \cite{6}.
We start by working on the femion algebra: 
$
 [  \psi_{m} ,  \psi_{n}  ]_{+} = \delta_{m+n}I   \quad  \textrm{and} \quad    \psi_{n} ^{\star} =  \psi_{-n}   \quad   (m, n \in \ZZZ + \1/2 ).
$
 As for $\WW_{1/2}$, we build its Verma module $H = \F_{NS}$ and the sesquilinear form $( . , . )$, which is a scalar product. $H$ is a prehilbert space, the unique unitary highest weight representation of the fermion algebra. Let the formal power series $\psi(z) = \sum_{n \in \ZZZ} \psi_{n+\frac{1}{2}}z^ {-n-1} $ called  fermion field. We inductively defined operators $D$ giving positive energy structure ($ \Leftrightarrow [D, \psi ] = z.\psi ' + \frac{1}{2} \psi$ ) and $T$ giving derivation ($ [T, \psi] = \psi' $). We compute $ (\psi(z) \psi(w)\Omega , \Omega ) =   \frac {1}{z-w}  $ ($ \vert z \vert >  \vert w \vert $), which permits to  prove inductively an anticommutation relation shortly written as: $\psi(z) \psi(w) = - \psi(w) \psi(z)$.
  We define this relation in a general framework as locality:  Let $H$  prehilbert space, and let  $A  \in (End H)[[z,z^{-1}]]$ formal power series  of the form $A(z) = \sum_{n \in \ZZZ} A(n)z^{-n-1} \ $ with $A(n) \in End(H) $  . Such fields $A $ and $B$ are  \textbf{local} if $\exists \varepsilon \in \ZZZ_{2}, \ \exists N \in \NNN $ such that $\forall c, d \in H$,
$\exists X(A,B,c,d) \in (z-w)^{-N} \CCC[z^{\pm 1}, w^{\pm 1}]$  such that:   \begin{center}
$
       X(A,B,c,d)(z,w) = \left\{  \begin{array}{c} (A(z) B(w)c , d )  \textrm{  \  if  }  \vert z \vert >  \vert w \vert    \\   (-1)^{\varepsilon}(B(w)A(z)c , d )  \textrm{  \  if  }  \vert w \vert >  \vert z \vert     \end{array}  \right. 
$  
\end{center}  Now, using locality and a contour integration argument, we can explicitly construct a field $A_{n}B$ from $A$ and $B$,  with $(A_{n}B)(m) =$ \begin{center} $ \left\{  \begin{array}{c} \sum_{p=0}^{n} (-1)^{p}C_{n}^{p}[A(n-p),B(m+p)]_{\varepsilon} \textrm{  \  if  }  n \ge 0   \\ \\   \sum_{p \in \NNN} C_{p-n-1}^{p}(A(n-p)B(m+p)-(-1)^{\varepsilon+n}B(m+n-p)A(p))  \textrm{  \  if  }  n < 0    \end{array}  \right. 
     $  \end{center}
     We obtain the operator product expansion (OPE) shortly written as \\   $  A(z)B(w) \sim  \sum_{n=0}^{N-1} \frac{(A_{n}B)(w)}{(z-w)^{n+1} }   $; and by an other contour integration argument: \begin{center} $
[A(m),B(n)]_{\varepsilon} = \left\{  \begin{array}{c} \sum_{p=0}^{N-1}  C_{m}^{p} (A_{p}B)(m+n-p) \textrm{  \  if  }  m \ge 0   \\  \\    \sum_{p=0}^{N-1} (-1)^{p} C_{p-m-1}^{p} (A_{p}B)(m+n-p) \textrm{  \  if  }  m < 0    \end{array}  \right.
$ \end{center} 
Thanks to Dong's lemma,  the operation $(A,B) \mapsto A_{n}B$ permits to generate many fields. To have a good behaviour, we define a system of generators as: 
$ \{A_{1}$,..., $ A_{r} \} \subset (EndH)[[z,z^{-1}]]$ with $ D, T \in End(H) $, $ \Omega \in H $ such that: \\  
 \textbf{(a)}   $\forall i,j \  A_{i}$ and $A_{j}$ are local with $N=N_{ij}$ and $\varepsilon= \varepsilon_{ij}=\varepsilon_{ii} . \varepsilon_{jj} $
\\  \textbf{(b)}   $\forall i \ [ T , A_{i}  ] = A'_{i} $
\\  \textbf{(c)}  $H = \bigoplus_{ n  \in \NNN + \frac{1}{2} } $ for $D$,  $dim( H_{n} )< \infty $  
\\  \textbf{(d)} $\forall i$  \  $[D, A_{i} ] = z.A_{i} ' + \alpha_{i} A_{i}$ with $\alpha_{i} \in \NNN+\frac{\varepsilon_{ii}}{2} $
\\  \textbf{(e)} $ \Omega \in H_{0} $,  $\Vert \Omega \Vert = 1$, and $\forall i \ \forall m \in \NNN $, $
A_{i}(m) \Omega = D \Omega = T\Omega= 0$
\\  \textbf{(f)}  $\A = \{ A_{i}(m), \forall i \ \forall m \in \ZZZ \}$ acts irreducibly on $H$, so that $<\A> . \Omega = H$
 \\ Hence, we generate a space $\S$, with $V:   H   \longrightarrow    \S  $ a state-field correspondence linear map.
$V(a)(z)$ is noted $V(a,z)$ and $V(a,z)\Omega_{\vert z=0} = a$. \\ 
Now,  $ \{ \psi  \} $ is a system of generator, we generate $\S$ and the map $V$ with  $ \psi (z)  = V(\psi_{-\1/2} \Omega, z) $; 
but, $\psi (z)\psi (w)  \sim     \frac{Id}{z-w} + 2 L(w) $, with $ L(z) =  \sum_{n \in \ZZZ} L_{n} z^{-n-2} =  \1/2 \psi_{-2}\psi (z)= V(\omega, z)   $ with $\omega  = \1/2 \psi_{-\3/2}\psi_{-\1/2}  \Omega $. Then, using OPE and Lie bracket, we find that $D = L_{0}$,  $T = L_{-1} $,  $L(z) L(w)  \sim     \frac{(c/2) Id }{(z-w)^{4}} +   \frac{2  L(w) }{(z-w)^{2}} +   \frac{ L'(w) }{(z-w)}$, and $[L_{m} , L_{n} ] = (m-n) L_{m+n}  +  \frac{c}{12} m(m^{2}-1) \delta_{m+n}$ with $c=  2 \Vert L_{-2} \Omega  \Vert ^{2}=\1/2$,  the central charge. ÊAs corollary, $\Vir$  acts on $H =  \F_{NS} $, and admits its unitary highest weight representation $L(c,h) = L( \1/2 , 0)$  as minimal submodule containing $ \Omega$.
We call $\omega \in H$ the Virasoro vector, and $L$  the Virasoro field. \\ We are now able to define  vertex operators superalgebras in general. \\ 
A vertex operator superalgebra is an $( H, V,  \Omega, \omega) $  with: 
\\  \textbf{(a)}  $H = H_{\bar{0}} \oplus  H_{\bar{1}} $ a prehilbert superspace.
\\  \textbf{(b)} $ V:    H    \rightarrow      (End H)[[z,z^{-1}]] $ a  linear map.
\\  \textbf{(c)} $ \Omega, \hspace{0,065cm} \omega \in H$  the  vacuum and Virasoro vectors. \\ 
 Let $\S_{\varepsilon} = V (H_{\varepsilon} )$,  $\S = \S_{\bar{0}} \oplus \S_{\bar{1}}$ and  $  A(z) = V(a ,z) = \sum_{n \in \ZZZ}A(n)z^{-n-1} $,  \\
 then $( H, V, \Omega, \omega) $ satisfies the followings axioms: 
\\  \textbf{(1)} $\forall  n \in \NNN$, $\forall A \in \S$, $A(n) \Omega = 0$,   $V(a,z)\Omega_{\vert z=0} = a $,   and  $V(\Omega , z) = Id $  
\\  \textbf{(2)}  $\A = \{ A(n) \vert  A \in \S ,  n \in  \ZZZ \}$ acts irreducibly on $H$,  so that $\A . \Omega = H$.
\\  \textbf{(3)} $\forall A \in \S_{\varepsilon_{1}}$, $\forall B \in \S_{\varepsilon_{2}}$,  $A$ and $B$ are local  with $\varepsilon = \varepsilon_{1} . \varepsilon_{2} $, $A_{n}B \in \S_{\varepsilon_{1} +  \varepsilon_{2}}$  	
\\  \textbf{(4)} $V(\omega , z) = \sum_{n \in \ZZZ}L_{n} z^{-n-2} $,  $  [L_{m} , L_{n} ] = (m-n) L_{m+n}  +  \frac{\Vert  2\omega \Vert^{2}}{12} m(m^{2}-1) \delta_{m+n}   $  
\\  \textbf{(5)} $H=\bigoplus_{ n  \in \NNN + \frac{1}{2} } H_{n} $ for $L_{0}$, $dim( H_{n} )< \infty $ and  $ H_{\varepsilon } = \bigoplus_{ n \in  \NNN + \frac{\varepsilon }{2}} H_{n} $  
\\ \textbf{(6)} $[L_{0}, V(a,z) ] = z.V'(a,z)  + \alpha.V(a,z)$  for $ a \in H_{\alpha}$
\\ \textbf{(7)} $[ L_{-1}, V(a,z) ] = V'(a,z) = V( L_{-1}.a , z)      \in \S   $ \\   
 As corollaries, we have that a system of generators, generating a Virasoro field $L \in \S $, with $D = L_{0}$ and $T = L_{-1}$, generates a vertex operator superalgebra;  the fermion field $\psi$ and  the Virasoro field $L$ generate one, each; and we prove the Borcherds associativity:  $V(a,z)V(b,w)  = V(V(a, z-w)b,w) $.
 
\subsection{Vertex $\gg$-superalgebras  and modules} \label{c=d?}  
Let $\gg$ be a simple Lie algebra of dimension $N$, a basis $( X_{a} )$, well normalized (see remark \ref{casimir}),   such that 
$ [X_{a} , X_{b} ] = i \sum_{c} \Gamma_{ab}^{c} X_{c}$    with   $\Gamma_{ab}^{c}  \in \RRR$ totally antisymmetric. 
Let its dual coxeter number $g = \frac{1}{4}\sum_{a,c} (\Gamma_{ac}^{b})^{2} $:  \begin{center}
\begin{tabular}{|c|c|c|c|c|c|c|c|c|c|}
\hline
  $\gg$ & $A_{n}$ & $B_{n}$  & $C_{n}$  & $D_{n}$  & $ E_{6}$ & $E_{7}$  & $E_{8}$ & $F_{4}$ & $G_{2}$    \\
  \hline
  $dim(\gg)$ & $n^{2} +  2n$ &$2n^{2} + n $&$2n^{2} + n $&$2n^{2} - n$ &$78$ & $133$&$ 248$ &$52$ &$14 $ \\
  \hline
   $g$ & $n+1$  & $2n-1$  & $n+1$  & $2n-2$  &$12$ & $18$&$30$ &$9$ & $4$ \\
\hline
\end{tabular}
\end{center}  For example,  $\gg = A_{1} =  \sl_{2}$, $dim(\gg) = 3$ and $g=2$. \\
Let $\widehat{\gg} _{+}$ the $\gg$-boson algebra: $ [X_{m}^{a} , X_{n}^{b}] = [ X_{a} , X_{b}]_{m+n} + m\delta_{ab}\delta_{m+n}.\L $, unique central extension (by $\L$) of the loop algebra $L \gg = C^{\infty}(\SSS^{1} , \gg)$  (see \cite{1} p 43). The unitary highest weight representations of $\widehat{\gg} _{+}$ are $H=L(V_{\lambda} , \ell)$, with $ \ell \in \NNN$ such that $\L \Omega  = \ell \Omega $ (the level of $H$), $H_{0} = V_{\lambda}$ irreducible representation of $\gg$ such that $(\lambda , \theta) \le \ell$   with $\lambda$ the highest weight and $ \theta$ the highest root (see \cite{1} p 45).
 The category $\CCCC_{\ell} $ of  representations for fixed $\ell$ is finite.
For example $\gg =  \sl_{2}$, $H=L(j , \ell)$, with $V_{\lambda} = V_{j}$ representations  of spin $j \le \frac{\ell}{2}$. 

We define the  $\gg$-fermion algebra   $\widehat{\gg} _{-}$ and the fermion fields, composed by $N$ fermions; and as for $N=1$, we generate a vertex operator superalgebra, but now, it contains  $\gg$-boson fields $(S^{a})$ whose related algebra is represented with $L(V_{0} , g)$; and thanks to $\widehat{\gg} _{-}$vertex background, the fields $(S^{a})$ generate a vertex operator superalgebra; by this way, we are  able to generate one, from $\widehat{\gg} _{+}$ and  $H=L(V_{0} , \ell)$, $\forall \ell \in \NNN$. Remark that because of the vacuum axiom,  the vertex structure need to take $V_{\lambda} = V_{0}$ trivial representation; in general, we have vertex modules (see further).  

Now, let $\widehat{\gg} = \widehat{\gg}_{+} \ltimes \widehat{\gg}_{-} $ the $\gg$-supersymmetric algebra; we prove it admits $H = L(V_{\lambda}, \ell) \otimes \F_{NS}^{\gg}$ as unitary highest weight representations. We generate a vertex operator superalgebra, with a Virasoro field $L$, and also a SuperVirasoro field $G$, which gives the supersymmetry boson-fermion:  Let $B^{a} =  X^{a} + S^{a}$ boson fields of level $d = \ell + g $, then: 

\begin{center} $
  B^{a}(z)G(w) \sim   d^{\1/2} \frac{\psi^{a}(w)}{(z-w)^{2}}  \quad   \textrm{and}   \quad   \psi^{a}(z)G(w) \sim   d^{-\1/2} \frac{B^{a}(w)}{(z-w)} 
 $.  \end{center}
 
  Finally, from $H^{\lambda} = L(V_{\lambda}, \ell)\otimes \F_{NS}^{\gg}$, we define the vertex module $(H^{\lambda} , V^{\lambda} )$ over $(H^{0} , V , \Omega, \omega) $, and we prove that $ \Vir_{\1/2}$ acts unitarily on $H^{\lambda} $ and admits $L(c,h)$  as minimal submodule containing the cyclic vector  Ê$\Omega^{\lambda}$,  with \\   $c = \frac{3}{2}  \cdot \frac{\ell +\t1 g}{\ell + g} dim(\gg)$,  $h = \frac{c_{V_{\lambda}}}{2(\ell + g)}  $ and    $c_{V_{\lambda}}$ the Casimir number of $V_{\lambda}$.

\newpage

\section{The Neveu-Schwarz algebra}
\subsection{Witt superalgebras and representations}  
\begin{definition}    A Lie superalgebra is a $\ZZZ_{2}$-graded vector space $ \dd = \dd_{\bar{0}} \oplus   \dd_{\bar{1}} $, 
together with a graded Lie bracket $[ . , . ]: \dd \times  \dd \to \dd $, such that $[ . , . ]$ is a bilinear map with $[ \dd_{i} , \dd_{j}] \subseteq \dd_{i+j}$,
and for homogeneous elements \\ $X \in   \dd_{x}$,   $Y \in   \dd_{y}$,  $Z \in   \dd_{z}$ : 
\begin{itemize}
\item  $[ X , Y ] = -(-1)^{xy}[ Y , X ] $  
 \item $(-1)^{xz}[X,[Y,Z]]+(-1)^{xy}[Y,[Z,X]]+(-1)^{yz}[Z,[X,Y]] = 0$ 
  \end{itemize}   \end{definition}
\begin{definition}  The Witt algebra $\WW$ is the Lie $\star$-algebra of vector fields on the circle,  generated by $d_{n}= ie^{i\theta n} \frac{d}{d\theta}$ $( n \in \ZZZ  )$. \end{definition}
\begin{remark} $\WW$ admits two supersymmetrics extensions, $\WW_{0}$ the Ramond sector (R) and $\WW_{1/2}$ the Neveu-Schwarz sector (NS) ((see \cite{kacv}, \cite{24b} chap 9). \end{remark}
Here, we trait only the (NS) sector. 
 \begin{definition}  Let $\dd = \WW_{1/2}$ the Witt superalgebra with: \begin{displaymath}   \left\{  \begin{array}{l} 
\lbrack d_{m},d_{n} \rbrack  = (m-n)d_{m+n} \\ 
\lbrack \gamma_{m},d_{n} \rbrack  = (m-\frac{n}{2})\gamma_{m+n}  \\ 
\lbrack \gamma_{m},\gamma_{n} \rbrack_{+} = 2d_{m+n}
\end{array}   \right.   \end{displaymath} 
together with the $\star$-structure,   $d^{\star}_{n} =d_{-n}$ and  $\gamma^{\star}_{m} = \gamma_{-m}$,  and\\Ê
$\textrm{ the super-structure: }  \dd_{\bar{0}} = \WW =  \bigoplus_{n \in \ZZZ}\CCC d_{n}, \ \  \dd_{\bar{1}} =  \bigoplus_{m \in \ZZZ + 1/2}\CCC \gamma_{m} $ \end{definition} 
 Now we investigate representations $\pi$  of $\WW_{1/2}$, which are :
 \begin{definition} Let $H$ be a  prehilbert  space.
      \label{proj}
\begin{description}
\item[(a)] Unitary:  $\pi (A)^{\star} = \pi (A^{\star})$ 
\item[(b)] Projective:  $A \mapsto \pi (A)$ is linear and $[\pi (A) , \pi (B)] - \pi ([A,B]) \in \CCC$.
\item[(c)] Positive energy : $H$ admits an orthogonal decomposition $ H = \bigoplus_{n \in  \frac{1}{2}\NNN} H_{n}$
 such that,  $ \exists D  $  acting on $H_{n}$ as multiplication by $n$,   $H_{0} \ne \{ 0 \} $ and  dim$(H_{n}) < + \infty $. Here,  $\exists h \in \CCC $  such that  $D   = \pi (d_{0}) -  hI.  $ 
\end{description}   \end{definition}
  
 \subsection{Investigation}  
  \begin{definition} Let $b: \WW_{1/2} \times \WW_{1/2} \to \CCC $ be the bilinear map defined by 
 \begin{displaymath}     [\pi (A) , \pi (B)] - \pi ([A,B]) = b(A,B)I  \ \  \textrm{ ($b$ is a $2$-cocycle) } 
\end{displaymath} \end{definition}
 \begin{definition}   Let $f : \WW_{1/2} \to \CCC $ be a $\star$-linear form. \\   $\partial f = (A,B) \mapsto  f([A,B])  $ is a 2-coboundary. \end{definition}
 \begin{remark}  $A \mapsto \pi (A) + f(A)I$ define also a projective, unitary, positive energy representation, where $b(A,B)$ becomes 
$b(A,B)-f([A,B])$. \end{remark}
\begin{proposition}  \label{suvi} (SuperVirasoro extension)   \ \ 
$\WW_{1/2}$ has a unique central extension, up to equivalent, i.e. $H_{2}(\WW_{1/2}, \CCC)$ is $1$-dimensional.  This extension admits  the basis   
$(L_{n})_{n \in \ZZZ}$, $(G_{m})_{m \in \ZZZ + \1/2}$,  $C$ central, with  $L_{n}^{\star} = L_{-n} $,  $G_{m}^{\star} = G_{-m} $, \\   $C=cI$,  $c \in \CCC$ called the \textbf{central charge};   and relations: 

\begin{displaymath} \left\{  \begin{array}{l} 
\lbrack L_{m},L_{n} \rbrack  = (m-n)L_{m+n} +\frac{C}{12}(m^{3} - m) \delta_{m+n}  Ê \\
\lbrack G_{m},L_{n} \rbrack = (m-\frac{n}{2})G_{m+n}    \\
\lbrack G_{m},G_{n}\rbrack_{+} = 2L_{m+n} + \frac{C}{3}(m^{2}- \frac{1}{4})\delta_{m+n} 
\end{array}   \right.    \end{displaymath}
\end{proposition} 
 \begin{proof} $ \begin{array}{l}  \end{array} $ \\  Let $L_{n} = \pi (d_{n})$ and $G_{m} = \pi (\gamma_{m})$ then: \\
$  \left.  \begin{array}{l}
\lbrack L_{m},L_{n} \rbrack = (m-n)L_{m+n} + b(d_{m} , d_{n})I \\
\lbrack G_{m},L_{n} \rbrack = (m-\frac{n}{2})G_{m+n} + b(\gamma_{m} , d_{n})I \\
\lbrack G_{m},G_{n}\rbrack_{+} = 2L_{m+n} + b(\gamma_{m} , \gamma_{n})I
\end{array}   \right. $ \\
In particular: \\  $  \left.  \begin{array}{l}
\lbrack L_{0},L_{n}\rbrack= -nL_{n} + b(d_{0} , d_{n})I \\
\lbrack L_{0},G_{n}\rbrack = -nG_{n} + b( d_{0} , \gamma_{n})I \\
\lbrack L_{1},L_{-1}\rbrack = 2L_{0} + b(d_{1} , d_{-1})I
\end{array}   \right. $ \\
We choose :  \\  $  \left. \begin{array}{l}
f(d_{n})=-n^{-1}b(d_{0} , d_{n})  \\
f(\gamma_{m})=-m^{-1}b( d_{0} , \gamma_{m})  \\
f(d_{0})=\frac{1}{2}b(d_{1} , d_{-1})
\end{array}   \right. $ \\
Then, after adjustment by $f$:  \\  $  \left. \begin{array}{l}
\lbrack L_{0},L_{n}\rbrack = -nL_{n}  \\
\lbrack L_{0},G_{n}\rbrack = -nG_{n}  \\
\lbrack L_{1},L_{-1}\rbrack = 2L_{0} 
\end{array}   \right. $  \newpage
Now  $D = L_{0} -  hI$ and if $v \in H_{k} $,  $Dv = kv$, then:  \\
   $DL_{n}v = L_{n} Dv + [D,L_{n}]v =  kL_{n} v + [L_{0},L_{n}]v = (k-n)L_{n}v$ \\
So, $L_{n} : H_{k} \to H_{k-n} \ ( =\{ 0 \}$  if $n>k$ ). \\
 Similary,  $G_{m} : H_{k} \to H_{k-m}$, then: \begin{displaymath} 
  \left\{  \begin{array}{l}
\lbrack L_{m},L_{n}\rbrack - (m-n)L_{m+n} : H_{m+n+k} \to H_{k}  \\
\lbrack G_{m},L_{n}\rbrack - (m-\frac{n}{2})G_{m+n}: H_{m+n+k} \to H_{k}   \\
\lbrack G_{m},G_{n}\rbrack_{+} - 2L_{m+n} : H_{m+n+k} \to H_{k} 
\end{array}   \right. \end{displaymath}
But  $ b(d_{m} , d_{n})I$, $b(\gamma_{m} , d_{n})I$, $b(\gamma_{m} , \gamma_{n})I : H_{m+n+k} \to H_{m+n+k} $, so: \begin{displaymath} 
 \left \{   \begin{array}{l}
 b(d_{m} , d_{n}) = A(m) \delta_{m+n}  \\
b(\gamma_{m} , d_{n}) = B(m) \delta_{m+n} = 0 \   \textrm{because} \  0 \notin \ZZZ+1/2 \ni m+n \\
b(\gamma_{m} , \gamma_{n}) = C(m) \delta_{m+n} 
\end{array}   \right. \end{displaymath}
Now, on  $\WW = \dd_{\bar{0}}$, $b(A,B) = -b(B,A) $, so, $A(m)= -A(-m)$ and $A(0) = 0$, 
and Jacobi identity implies $b([A,B],C)+b([B,C],A)+b([C,A],B) =0 $, 
then, for $d_{k}, d_{n}, d_{m}  $ with $k+n+m = 0$ :
 \begin{displaymath} (n-m)A(k)+(m-k)A(n)+(k-n)A(m) = 0  \end{displaymath}
Now, with $k=1 $ and $m=-n-1$,  $(n-1)A(n+1)=(n+2)A(n)-(2n+1)A(1)$. \\
Then $A(n)$ is completely determined by the knowledge of $A(1)$ and $A(2)$, and so, the solutions are a $2$-dimensional space. \\
Now, $n $ and $n^{3}$ are solutions, so $A(n) = a.n+b.n^{3}$ .\\
Finally, because $[ L_{1},L_{-1}] = 2L_{0} $,   $A(1)=0$ and $a+b=0$, we obtain:
 \begin{displaymath}     A(n) = b(n^{3} - n)= \frac{c}{12}(n^{3} - n), \quad  \textrm{$c \in \CCC$ the central charge}. \end{displaymath} 
 \textbf{Process \ref{suvi}.}\\ $[[A ,B ]_{+}, C  ] = [  A , [B,  C]  ] _{+}+[  B , [A,  C]  ] _{+}$  then: \\
 $[[ G_{r},G_{s}]_{+},L_{n}] = [ G_{r},[G_{s} , L_{n}]]_{+} + [ G_{s},[G_{r} , L_{n}]]_{+} $ \\
$= [ 2L_{r+s}, L_{n}] = [ G_{r},(s-\frac{1}{2}n)G_{n+s}]_{+} +[ G_{s},(r-\frac{1}{2}n)G_{n+r}]_{+}$ \\
$=2(r+s-n) L_{r+s+n} - \delta_{r+s+n} \frac{c}{6}(n^{3}-n) \\   = (s-\frac{1}{2}n)(2L_{r+s+n} + C(r) \delta_{r+s+n}) - (r-\frac{1}{2}n)(2L_{r+s+n} + C(s) \delta_{r+s+n})$ \\
Then taking $r+s+n = 0 $,   $\frac{c}{6}(n^{3}-n) + (s-\frac{1}{2}n)C(r) +  (r-\frac{1}{2}n)C(s) = 0$. \\
Finally, with $n=2s$ and $r=-3s$,   $C(s) =  \frac{c}{3}(s^{2}- \frac{1}{4})$. \end{proof} \newpage 
\begin{definition}  The central extension of $\WW_{1/2}$ is called  $\Vir_{1/2}$,  the SuperVirasoro algebra (on sector NS),  also called Neveu-Schwarz  algebra.  \end{definition}

\begin{theorem} (Complete reducibility) 
\begin{description}
\item[ (a) ] If $H$ is a unitary, projective, positive energy representation of $\WW_{1/2}$,
 then any non-zero vector $v$ in the lowest energy subspace $H_{0}$ generates an irreducible submodule.
 \item[ (b) ]  $H$ is an orthogonal direct sum of irreducibles such representations.
  \end{description}  \end{theorem}
\begin{proof} 
 (a)  Let $K$ be the minimal  $\WW_{1/2}$-submodule containing $v$. \\ 
 Clearly, since $L_{n}v=G_{m}v = 0$ for $m,n > 0$ and $L_{0}v = hv$, we see that $K$ is spanned by all products $R.v$ with :
 \begin{displaymath} R = G_{-j_{\beta}} \dots G_{-j_{1}}L_{-i_{\alpha}} \dots L_{-i_{1}}, \quad   Ê0 <  i_{1} \le \dots \le i_{\alpha},   \quad  \frac{1}{2} \le j_{1} < \dots < j_{\beta}
\end{displaymath}
  But then, $K_{0} = \CCC v $. Let $K'$ be a submodule of $K$, and let $p$ be the orthogonal projection onto $K'$.
  By unitarity, $p$ commutes with the action of $\WW_{1/2} $, and hence with $D$. Thus $p$ leaves  $K_{0} = \CCC v $ invariant, so $pv=0$ or $v$. \\
  But  $pRv = Rpv  $,  hence $K'= {0}$ or $K$ and $K$ is irreducible. \\ \\
  (b) Take the irreducible module $M_{1}$ generated by a vector of lowest energy. Now (changing $h$ into $h'=h+m$ if necessary),we repeat this process for $M^{\bot}_{1}$, to get  $M_{2}, M_{3}, \dots$ 
    The positive energy assumption shows that $H = \oplus M_{i} $  \end{proof}   
  \begin{theorem}(Uniqueness)   If $H$ and $H'$ are irreducibles   with $c=c'$ and $h=h'$, then they are unitarily equivalents as  $\WW_{1/2} $-modules.  \end{theorem} 
   \begin{proof} $H_{0}= \CCC u$ and  $H'_{0}= \CCC u'$ with $u,u'$ unitary. \\ 
    Let $ U : H \to H' , Au \mapsto Au'$, we want to prove that $U^{\star}U=UU^{\star}=Id$. \\ 
   Let $Au \in H_{n}$, $Bu \in H_{m}$: \\
   If $n \neq m$, for example, $n<m$, then $B^{\star}Au \in H_{n-m}= {0}$ and \\  $(Au,Bu) = (B^{\star}Au , u) = 0 = (Au', Bu')$. \\
 If $n = m $, then $D = B^{\star}A$  is a constant energy operator, so in $  \CCC L_{0} \oplus \CCC C $. \\
 Now, $(L_{0}u,u)=h=(L_{0}u',u')$ iff $h=h'$ and $(Cu,u)=c=(Cu',u')$ iff $c=c'$.
 Finally,  $(v,w) = (Uv,Uw) \  \forall v,w \in H$ and $(v',w') = (U^{\star}v',U^{\star}w') \  \forall v',w' \in H'$  iff  $h=h'$ and $c=c'$. \\
 So,  $U^{\star}U=UU^{\star}=Id$, ie, $H$ and $H'$ are unitarily equivalents. \end{proof}
 \begin{definition}  \label{changer}  $\Vir_{1/2} = \Vir^{-}_{1/2} \oplus \Vir^{0}_{1/2} \oplus \Vir^{+}_{1/2}  \ \ \   \mathrm{with}  \ \ \    \Vir^{0}_{1/2}= \CCC L_{0} \oplus \CCC C$
  \begin{displaymath}
     \Vir^{+}_{1/2}= \bigoplus_{m,n>0}\CCC L_{m} \oplus \CCC G_{n}  \ \ \ \ \   \Vir^{-}_{1/2}= \bigoplus_{m,n<0}\CCC L_{m} \oplus \CCC G_{n} 
\end{displaymath}      \end{definition}   \begin{remark} 
This decomposition pass to the universal envelopping : 
 \begin{displaymath}
 \U (\Vir_{1/2} ) = \U (\Vir^{-}_{1/2} ) \cdot  \U (\Vir^{0}_{1/2}) \cdot  \U (\Vir^{+}_{1/2})
\end{displaymath}       \end{remark} 
\begin{remark}We see that an  irreducible, unitary, projective, positive energy representation of  $\WW_{1/2} $ is exactly given by a unitary highest weight representation of $\Vir_{1/2} $ (see the following section).  \end{remark}

\subsection{Unitary highest weight representations} \label{unit}  
\begin{definition}   Let the Verma module $ H= V(c,h)$ be the $\Vir_{1/2} $-module freely generated by followings conditions:
\begin{description}
\item[(a)]  $\Omega \in H$,  called the cyclic vector   ($\Omega \neq 0$).
\item[(b)]  $L_{0}\Omega = h\Omega$,   $C\Omega = c\Omega   \ \ ( h,c \in \RRR) $
\item[(c)]  $\Vir^{+}_{1/2}\Omega = \{0 \}$
\end{description}    \end{definition}
\begin{lemma}$ \U (\Vir^{-}_{1/2} )\Omega  =  H $ and a set of generators is given by: \\  
$G{-j_{\beta}} \dots G_{-j_{1}}L_{-i_{\alpha}} \dots L_{-i_{1}}\Omega, \quad   Ê0 <  i_{1} \le \dots \le i_{\alpha},   \quad  \frac{1}{2} \le j_{1} < \dots < j_{\beta}$ \end{lemma}
\begin{proof} It's clear. \end{proof}
\begin{lemma}  $V(c,h)$ admits a canonical sesquilinear form $( . , . ) $, \\ completely  defined by: 
\begin{description}
\item[(a)]  $(\Omega , \Omega )= 1$
\item[(b)]  $\pi (A)^{\star} = \pi(A^{\star})$
\item[(c)]  $(u,v) = \overline{(v,u)} \ \  \forall u, v \in H$    (in particular  $(u,u)=  \overline{(u,u)}  \in \RRR  $). 
\end{description} \end{lemma}
\begin{proof} It's clear. \end{proof}

\begin{definition}
$u \in V(c,h)$ is a ghost if $(u,u) <  0 $. \end{definition}
\begin{lemma} If $ V(c,h)$ admits no ghost then $ c,h \ge 0 $  \end{lemma}
\begin{proof}   Since $L_{n}L_{-n} \Omega =  L_{-n}L_{n} \Omega +2nh \Omega +c \frac{n(n^{2}-1)}{12}\Omega$,
  \\Ê we have $(L_{-n} \Omega,L_{-n} \Omega)  = 2nh + \frac{n(n^{2}-1)}{12}c \ge 0$.
   \\  Now, taking $n$ first equal to $1$ and then very large, we obtain the lemma.  \end{proof}
\begin{definition}  \label{Lhc} Let $K(c,h) = ker(. , .) = \{ x \in V(c,h) ; (x,y) = 0 \  \forall y   \}$ \\ the maximal proper submodule of $V(c,h)$,      
and $L(c,h)= V(c,h)/K(c,h)$, irreducible highest weight representation of $\Vir_{1/2}$, with $(. , .)$ well-defined on. \end{definition}              
\begin{definition}  $ u \in V(c,h)$ is a null vector if $(u,u) = 0 $. \end{definition}
\begin{lemma} On no ghost case,  the set of null vectors is  $ K(c,h)$.   \end{lemma}
\begin{proof}   Let $x$ be a null vector, and $y \in V(c,h)$. \\ 
 By assumption $\forall \alpha , \beta \in \CCC $,  $(\alpha x + \beta y ,  \alpha x + \beta y) \ge 0$.  
  We develop it, with $\alpha = (y,y) $ and $ \beta = - (x,y) $,  we obtain :  $ \vert (x,y) \vert^{2} (y,y)  \le  (x,x)(y,y)^{2} = 0$.  \\ 
  So if $y$ is not a null vector then $(x,y) = 0 $. 
  Else, $(x,x)= (y,y)=0$, so taking  $\alpha = 1 $ and $ \beta = - (x,y) $, we obtain $ 2\vert (x,y) \vert^{2}   \le 0 $ and so $(x,y) = 0  $  \end{proof}
\begin{corollary} $L(c,h)$ is a  unitary highest weight representation. \end{corollary}
\begin{proof} Without ghost,   $( . , . )$ is a scalar product on  $L(c,h)$. \end{proof}
 \begin{remark} Theorem 1.2 of \cite{NSOAII} will be proved  classifying no ghost cases.  \end{remark}
 
\newpage      
\section{Vertex operators superalgebras }
We  give a progressive introduction to  vertex operators superalgebras structure.
We start with the fermion algebra as example. We work on to obtain, at the end of the section,  vertex axioms naturally.
\label{vosa}
\subsection{Investigation on fermion algebra}   \label{fermion} \label{sec1}
\begin{definition}
Let the fermion algebra (of sector NS),  generated by $( \psi_{n})_{ n \in \ZZZ + \frac{1}{2} } $,  \\ and $I$ central, with the relations:
\begin{displaymath}
 [  \psi_{m} ,  \psi_{n}  ]_{+} = \delta_{m+n}I    \quad  \textrm{and} \quad    \psi_{n} ^{\star} =  \psi_{-n}
 \end{displaymath}     \end{definition}
\begin{definition} (Verma module)  Let $H = \F_{NS}$ freely generated by: 
\begin{description}
\item[(a)]  $\Omega \in H$  is called the vacuum vector , $\Omega \neq 0$.
\item[(b)]  $\psi_{m}\Omega = 0 $  $\forall m > 0$ 
\item[(c)]  $ I \Omega = \Omega$
\end{description}  \end{definition}
\begin{lemma} A set of generators of $H$  is given by:  \\ 
  $ \psi_{-m_{1}} \dots \psi _{-m_{r}}\Omega    \ \ \    m_{1} < \dots < m_{r}  \   \       r \in \NNN, \ Êm_{i} \in \NNN + \frac{1}{2}  $  \end{lemma}
  \begin{proof} It's clear.  \end{proof}
\begin{lemma} $H$ admits the sesquilinear form $(.,.)$ completely defined by : 
\begin{description}
\item[(a)] $ (\Omega , \Omega )= 1$
\item[(b)]  $(u,v) = \overline{(v,u)} \ \  \forall u, v \in H$
\item[(c)]  $(\psi_{n}u,v) = (u,\psi_{-n}v)    \ \  \forall u, v \in H \ \  $    ie  $\pi( \psi_{n})^{\star} = \pi( \psi_{n}^{\star})$
\end{description}
$(.,.)$ is a scalar product and $H$ is a prehilbert space. \end{lemma}  \begin{proof} It's clear.  \end{proof}
\begin{remark} $H$ is an irreducible representation of the fermion algebra. \\  It is its unique unitary highest weight
representation.  \end{remark}
\begin{remark} 
 $\psi_{n}^{2} = \frac{1}{2} [  \psi_{n} ,  \psi_{n}  ]_{+} = 0 $  if $n \neq 0$  \end{remark}
\begin{definition}(Operator $D$)  \label{D}  Let $D \in End(H)  $ inductively defined by :   
\begin{description}
\item[(a)]  $ D \Omega = 0 $
\item[(b)]  $D \psi_{-m}a =  \psi_{-m}Da + m \psi_{-m}a \ \ \ $   $\forall m \in \NNN + \frac{1}{2}   $  and $\forall a \in H$
\end{description}   \end{definition}
\begin{lemma} $D$ decomposes $H$ into $\bigoplus_{ n  \in \NNN + \frac{1}{2} }  H_{n}$  with $D \xi = n \xi $ \\   $\forall \xi \in H_{n}$,  $dim( H_{n} )< \infty $
 and  $ H_{n} \perp H_{m} $   if $n \neq m$  \end{lemma}  
 \begin{proof}  Let $a = \psi_{-m_{1}} \dots \psi _{-m_{r}}\Omega$ be a generic element of the base of $H$, \\Ê then $ D.a = ( \sum m_{i})a $.     \end{proof}
\begin{remark}   $[D,\psi_{m}] = -m \psi_{m}  $ and  $ \Omega \in H_{0}  $,  so   $\psi_{m} :  H_{m+n} \to  H_{n}$.  \end{remark}
\begin{definition}(Operator $T$)  Let $T \in End(H)  $ inductively defined by :   
\begin{description}
\item[(a)]  $ T \Omega = 0 $
\item[(b)]  $T \psi_{-m}a =  \psi_{-m}Ta + (m-\frac{1}{2}) \psi_{-m-1}a \ \ \ $   $\forall m \in \NNN + \frac{1}{2}   $  and $\forall a \in H$
\end{description}   \end{definition}
\begin{remark}   $[T,\psi_{m}] = -(m-\frac{1}{2}) \psi_{m-1} $.  \end{remark}
\begin{definition} Let $\psi(z) = \sum_{n \in \ZZZ} \psi_{n+\frac{1}{2}}.z^ {-n-1} $ the fermion operator.  \end{definition}
\begin{remark}  $\psi \in (End H) [[z,z^{-1}]] $  is a formal power series.  \end{remark}
\begin{lemma} (Relations with $\psi_{n}$,  $D$ and $T$) 
\begin{description} 
\item[(a)]  $[\psi_{m+\frac{1}{2}}, \psi ]_{+} = z^{m}$
\item[(b)]   $[D, \psi ] = z.\psi ' + \frac{1}{2} \psi$
\item[(c)]     $ [T, \psi] = \psi ' $
\end{description}   \end{lemma}
\begin{proof} $[\psi_{m+\frac{1}{2}}, \psi (z)]_{+} = \sum [\psi_{m+\frac{1}{2}}, \psi_{n+\frac{1}{2}}]_{+} .z^{-n-1} = z^{m}$  \\
$[D, \psi (z)] = \sum (-n-\frac{1}{2})\psi_{n+\frac{1}{2}}.z^{-n-1} = z.\psi '(z) + \frac{1}{2} \psi (z) $ \\  
$[T, \psi (z)] = \sum (-n)\psi_{n-\frac{1}{2}}.z^{-n-1} = \sum (-n-1)\psi_{n+\frac{1}{2}}.z^{-n-2} = \psi '(z)$ 
  \end{proof}
   \begin{remark} 
   $(.,.)$ induces  $(\psi(z_{1}) ... \psi(z_{n})c,d) \in \CCC [[z_{1}^{\pm 1},...,z_{n}^{\pm 1}]]$, $\forall c, d \in H$.  \end{remark} 
\begin{lemma}  \label{manifestation} 
$(\psi(z) \Omega , \Omega ) = 0  $ \ \  and \ \  
  $ (\psi(z) \psi(w)\Omega , \Omega ) =   \frac {1}{z-w}  \   $   if  $ \vert z \vert >  \vert w \vert $  .
\end{lemma}
\begin{proof}  
$(\psi(z) \Omega , \Omega ) = \sum_{n \in \ZZZ} (\psi_{n+\1/2} \Omega , \Omega ).z^{-n-1} = 0  $ \\ 
 $ (\psi(z) \psi(w)\Omega , \Omega ) =  \sum_{m,n \in \ZZZ} (\psi_{m+\1/2} \Omega , \psi_{-n-\1/2}\Omega ).z^{-n-1} w^{-m-1} 
 \\  =  \sum_{m,n \in \ZZZ} (\psi_{m-\1/2} \Omega , \psi_{-n-\1/2}\Omega ).z^{-n-1} w^{-m} =  \sum_{n \in \NNN} (\psi_{-n-\1/2} \Omega , \psi_{-n-\1/2}\Omega ).z^{-n-1} w^{n} 
\\   =   z^{-1} \sum_{n \in \NNN} (\frac{w}{z})^{n} =  \frac {1}{z-w}  \   $   if  $ \vert z \vert >  \vert w \vert $   
\end{proof}
\begin{lemma} $\forall c, d \in H$,      \label{pour}
$(\psi(z) c , d )  \in \CCC[z,z^{-1}] $.
 \end{lemma}
   \begin{proof}   $(\psi(z) \psi_{-n-\1/2}c , d )  = (c,d). z^{-n-1}  -  (\psi(z)c , \psi_{n+\1/2}d ) $ \\   
   $(\psi(z)c, \psi_{-n-\1/2}d )  = (c,d). z^{n}  -  (\psi(z)\psi_{n+\1/2}c,d ) $  \\
 Then, the result follows by lemma  \ref{manifestation} and induction.      \end{proof}
 \begin{proposition} $\forall c, d \in H$, 
 $\exists X(c,d) \in (z-w)^{-1}\CCC[z^{\pm 1}, w^{\pm 1}]$  such that:   
 \begin{displaymath}
      X(c,d)(z,w) = \left\{  \begin{array}{c} (\psi(z) \psi(w)c , d )  \textrm{  \  if  }  \vert z \vert >  \vert w \vert    \\   -(\psi(w) \psi(z)c , d )  \textrm{  \  if  }  \vert w \vert >  \vert z \vert     \end{array}  \right. 
\end{displaymath}   
 \end{proposition}
  \begin{proof}  $ (\psi(z) \psi(w)  \psi_{-n-\1/2}c , d ) = (\psi(z) c , d ) w^{-n-1} -  (\psi(w) c , d ) z^{-n-1} + (\psi(z) \psi(w) c , \psi_{n+\1/2} d )  $  \\ 
  $ (\psi(z) \psi(w) c ,  \psi_{-n-\1/2}d ) = (\psi(w) c , d ) z^{n} -  (\psi(z) c , d ) w^{n} + (\psi(z) \psi(w)  \psi_{n+\1/2}c , d ) $ \\ 
 Then, the result follows by lemma  \ref{manifestation}, \ref{pour}, symmetry and induction.        
         \end{proof} 
         
\subsection{General framework} 

\begin{definition} \label{localocal} Let $ H$  prehilbert and $A \in (End H)[[z,z^{-1}]]$ a formal power series defined as 
 $A(z) = \sum_{n \in \ZZZ} A(n)z^{-n-1} \ $ with $A(n) \in End(H) $.  \end{definition}
\begin{definition}  \label{local} Let  $ A, B   \in (End H)[[z,z^{-1}]]$  \\ 
$A $ and $B$ are  \textbf{local} if $\exists \varepsilon \in \ZZZ_{2}, \ \exists N \in \NNN $ such that $\forall c, d \in H$:  \\
$\exists X(A,B,c,d) \in (z-w)^{-N} \CCC[z^{\pm 1}, w^{\pm 1}]$  such that:   
 \begin{displaymath}
       X(A,B,c,d)(z,w) = \left\{  \begin{array}{c} (A(z) B(w)c , d )  \textrm{  \  if  }  \vert z \vert >  \vert w \vert    \\   (-1)^{\varepsilon}(B(w)A(z)c , d )  \textrm{  \  if  }  \vert w \vert >  \vert z \vert     \end{array}  \right. 
\end{displaymath}   \end{definition}
\begin{example}  \label{ex}  $ \psi $ is local with itself, with $N=1$ and $\varepsilon  =  \bar{1}  $  \end{example}
\begin{notation}  $ [  X  , Y  ]_{\varepsilon} = \left\{  \begin{array}{c} XY-YX  \textrm{  \  if  }  \varepsilon = \bar{0}   \\    XY+YX  \textrm{  \  if  }  \varepsilon = \bar{1}   \end{array}  \right.  $ \end{notation}  
\begin{remark}   \label{trivial}
Let $n \in \NNN$,  then, $(z-w)^{n} = \sum_{p=0}^{n}  C_{n}^{p} (-1)^{p}w^{p}z^{n-p} $   and,  \\ 
$(z-w)^{-n} = \left\{  \begin{array}{c} \sum_{p \in \NNN} C_{p+n-1}^{p}w^{p}z^{-p-n}  \textrm{  \  if  }  \vert z \vert >  \vert w \vert    \\    (-1)^{n} \sum_{p \in \NNN} C_{p+n-1}^{p}z^{p}w^{-p-n}   \textrm{  \  if  }  \vert w \vert >  \vert z \vert     \end{array}  \right.  $ \end{remark}

\begin{proposition}  \label{investig} 
   Let $A, B$ local and $c, d \in H$ then: \\ 
   $X(A,B,c,d)(z,w)=\sum_{n \in \ZZZ}X_{n}(A,B,c,d)(w)(z-w)^{-n-1} $,  \\
   $X_{n}(A,B,c,d)(w) = (A_{n}B(w)c,d) $, \\ 
   $A_{n}B(w)= \sum_{m \in \ZZZ} (A_{n}B)(m)w^{-m-1}  $  and 
$(A_{n}B)(m) =$ \begin{displaymath} \left\{  \begin{array}{c} \sum_{p=0}^{n} (-1)^{p}C_{n}^{p}[A(n-p),B(m+p)]_{\varepsilon} \textrm{  \  if  }  n \ge 0   \\ \\   \sum_{p \in \NNN} C_{p-n-1}^{p}(A(n-p)B(m+p)-(-1)^{\varepsilon+n}B(m+n-p)A(p))  \textrm{  \  if  }  n < 0    \end{array}  \right. 
     \end{displaymath} 
     \end{proposition} 
     $\begin{array}{l}  \end{array} $
 \begin{proof}  
$X(A,B,c,d) \in \CCC[z^{\pm 1}, w^{\pm 1}, (z-w)^{-1}]$, we  develop it around $z=w$:   \\
$X(A,B,c,d)(z,w)=\sum_{n \in \ZZZ}X_{n}(A,B,c,d)(w)(z-w)^{-n-1} $  \\
 with $X_{n}(A,B,c,d)(w) = \frac{1}{2 \pi i} \oint_{w}(z-w)^{n}X(A,B,c,d)(z,w)dz $. \\  By contour integration argument 
($  \oint_{w} = \int_{\vert z \vert = R > \vert w \vert} - \int_{\vert z \vert = r < \vert w \vert}$),   we obtain: \\  \\ 
$X_{n}(A,B,c,d)(w) = \frac{1}{2 \pi i} (\int_{\vert z \vert = R > \vert w \vert} - \int_{\vert z \vert = r < \vert w \vert})(z-w)^{n}X(A,B,c,d)(z,w)dz $ \\
$ = \frac{1}{2 \pi i}\int_{\vert z \vert = R > \vert w \vert}(z-w)^{n}(A(z) B(w)c , d )dz -  \frac{(-1)^{\varepsilon}}{2 \pi i} \int_{\vert z \vert = r < \vert w \vert}(z-w)^{n} (B(w)A(z)c , d ) dz \\ 
=  \frac{1}{2 \pi i} \sum_{q \in \ZZZ,  p=0}^{n}(\int_{\vert z \vert = R > \vert w \vert} C_{n}^{p} (-1)^{p}z^{n-p}w^{p}(A(q) B(w)c , d )z^{-q-1}dz  \\ -(-1)^{\varepsilon} \int_{\vert z \vert = r < \vert w \vert}C_{n}^{p} (-1)^{p}z^{n-p}w^{p}(B(w)A(q) c , d )z^{-q-1}dz) \\ =  (\sum_{p=0}^{n} (-1)^{p}w^{p}C_{n}^{p}[A(n-p),B(w)]_{\varepsilon}c ,d)$,   \   with $n \in \NNN$. \\ \\ 
$X_{-n}(A,B,c,d)(w) =  \frac{1}{2 \pi i} \sum_{q \in \ZZZ,  p \in \NNN}(\int_{\vert z \vert = R > \vert w \vert} C_{p+n-1}^{p} z^{-n-p}w^{p}(A(q) B(w)c , d )z^{-q-1}dz \\Ê - (-1)^{\varepsilon} \int_{\vert z \vert = r < \vert w \vert}C_{p+n-1}^{p} (-1)^{n}w^{-n-p}z^{p}(B(w)A(q) c , d )z^{-q-1}dz) \\  =(\sum_{p \in \NNN} C_{p+n-1}^{p}(w^{p}A(-n-p)B(w)-(-1)^{\varepsilon+n}w^{-n-p}B(w)A(p))c ,d)$ 
    \end{proof}
 \begin{definition} Let the operation $(A,B) \to A_{n}B$ as for proposition \ref{investig}.   \end{definition}  
  \begin{formula}   \label{formula}  
The  formula of  $(A_{n}B)(m)$ on proposition \ref{investig}. \end{formula}   

\begin{corollary}(Operator product expansion)  \label{OPE} Let $A, B$ local, and  $c, d \in H$:
 \begin{displaymath}   (A(z)B(w)c,d) \sim ( \sum_{n=0}^{N-1} \frac{(A_{n}B)(w)}{(z-w)^{n+1} }c,d) \ \  \   \textrm{near } z=w      \end{displaymath}   \end{corollary}
 \begin{proof}  $X(A,B,c,d)(z,w) = \sum_{n \in \ZZZ} (A_{n}B)(w)c,d)(z-w)^{-n-1} \\   \in (z-w)^{-N} \CCC[z^{\pm 1}, w^{\pm 1}]$,   so $A_{n}B = 0$ for $-n-1<-N$ ie $n \ge N $.  \end{proof}
 \begin{remark} We  write OPE as:  $  A(z)B(w) \sim  \sum_{n=0}^{N-1} \frac{(A_{n}B)(w)}{(z-w)^{n+1} }   $.  \end{remark}
  \begin{remark}    $z^{m} = \left\{  \begin{array}{c} \sum_{k=0}^{m} C_{m}^{k}(z-w)^{k}w^{m-k} \textrm{  \  if  }  m \ge 0   \\    \sum_{k \in \NNN} (-1)^{k}C_{k-m-1}^{k}(z-w)^{k}w^{m-k}  \textrm{  \  if  }  m < 0    \end{array}  \right. $  \end{remark}
\begin{proposition} (Lie bracket )  \label{litre}  
Let $A, B$ local, with $\varepsilon \in \ZZZ_{2}$  then: 
\begin{displaymath} 
[A(m),B(n)]_{\varepsilon} = \left\{  \begin{array}{c} \sum_{p=0}^{N-1}  C_{m}^{p} (A_{p}B)(m+n-p) \textrm{  \  if  }  m \ge 0   \\  \\    \sum_{p=0}^{N-1} (-1)^{p} C_{p-m-1}^{p} (A_{p}B)(m+n-p) \textrm{  \  if  }  m < 0    \end{array}  \right.
 \end{displaymath}   \end{proposition}
\begin{proof} 
 $ \forall c, d \in H$, 
$([A(m),B(n)]_{\varepsilon}c,d) = \\Ê  \frac{1}{(2 \pi i)^{2}} (\int \int_{\vert z \vert = R > \vert w \vert} - \int  \int_{\vert z \vert = r < \vert w \vert})z^{m}w^{n}X(A,B,c,d)(z,w)dzdw$ \\ 

By contour integration argument 
($   \int \int_{\vert z \vert = R > \vert w \vert} - \int  \int_{\vert z \vert = r < \vert w \vert} =   \oint_{0} \oint_{w} $):  \\ \\  
$([A(m),B(n)]_{\varepsilon}c,d) =   \ \frac{1}{2 \pi i}  \oint_{0}  w^{n}  \frac{1}{2 \pi i}    \oint_{w}z^{m} ( \sum_{p=0}^{N-1} \frac{(A_{p}B)(w)}{(z-w)^{p+1} }c,d)dzdw $\\ \\ 
We suppose $m \ge 0$, then by previous remark, $([A(m),B(n)]_{\varepsilon}c,d) = \\Ê
=   \frac{1}{2 \pi i}  \oint_{0}  w^{n}  \frac{1}{2 \pi i}    \oint_{w}\sum_{k=0}^{m} C_{m}^{k}w^{m-k} ( \sum_{p=0}^{N-1} \frac{(A_{p}B)(w)}{(z-w)^{p+1-k} }c,d)dzdw  \\
  =   \frac{1}{2 \pi i}  \oint_{0} (  \sum_{p=0}^{N-1} w^{n+m-p}C_{m}^{p} (A_{p}B)(w)c,d)dw \\ 
   = \frac{1}{2 \pi i}  \oint_{0} (  \sum_{r \in \ZZZ  p=0}^{N-1} w^{n+m-p-r-1}C_{m}^{p} (A_{p}B)(r)c,d)dw  \\ 
   =   ( \sum_{p=0}^{N-1}  C_{m}^{p} (A_{p}B)(m+n-p) c , d) $  \ \ (we take  $C_{m}^{p} = 0$ if $p>m$ ). \\ \\
   Similary for $m<0$..., and the result follows.   \end{proof}
   \begin{formula}   \label{lie}
   The formula of $[A(m),B(n)]_{\varepsilon}$ on proposition \ref{litre}.   \end{formula}   
   
 \begin{definition} (Operator D) \  Let  $ D \in End(H)$ decomposing $H$ into $\bigoplus_{ n  \in \NNN + \frac{1}{2} }  H_{n}$  with $D \xi = n \xi $  $\forall \xi \in H_{n}$,  $dim( H_{n} )< \infty $
 and  $ H_{n} \perp H_{m} $   if $n \neq m$.    \end{definition}
 
 \begin{notation} 
Let $A'(z) = \frac{d}{dz} A(z)  = \sum_{n \in \ZZZ}(-n)A(n-1)z^{-n-1}$.
    \end {notation} 

\begin{definition} \label{L0}
$A \in (End H)[[z,z^{-1}]]$ is graded if: \\
 $ \exists \alpha \in \frac{1}{2} \NNN  $ such that  $ \lbrack D , A(z) \rbrack = zA'(z)+ \alpha A(z)  $
 \end{definition} 

 \begin{lemma} \label{1} $A$ is graded with $\alpha$  $\iff$ \\Ê  $A(n): H_{m} \to  H_{m-n+\alpha -1} \ \   \forall n \in \ZZZ ,  \forall m  \in \1/2\NNN$  
 \end{lemma} 
\begin{proof}  
$[D,A(z)]= zA'(z)+ \alpha A(z) = \sum_{n \in \ZZZ} (\alpha -1-n) A(n)z^{-n-1} \\
\iff   [D,A(n)] = (\alpha -1-n) A(n)   \ \   \forall n \in \ZZZ   \\
\iff  \forall n \in \ZZZ , \forall m  \in \1/2\NNN ,  \forall \xi \in H_{m}  \\ 
 DA(n) \xi = A(n)D \xi +   [D,A(n)]\xi = (m-n+\alpha -1)A(n) \xi  \\ 
\iff  A(n): H_{m} \to  H_{m-n+\alpha -1} \ \   \forall n \in \ZZZ ,  \forall m  \in \1/2\NNN$.
 \end{proof}   
 
  \begin{lemma} \label{2} Let $A$, $B$ local and graded with $\alpha$ and $\beta$  then: \\Ê  $  [D,A_{n}B(z)]= z(A_{n}B)'(z)+ ( \alpha + \beta -n-1) A_{n}B(z)$. 
 \end{lemma} 
\begin{proof}    $A(n): H_{m} \to  H_{m-n+\alpha -1}$ and    $B(n): H_{m} \to  H_{m-n+\beta -1}$ \\
Now, by formula \ref{formula},   $A_{p}B(n): H_{m} \to  H_{m-n+(\alpha + \beta - p-1) -1}$   \\
The result follows by the previous lemma. \end{proof} 

\begin{lemma} \label{3} Let $A$, $B \in (EndH)[[z,z^{-1}]]$,  graded with $\alpha$ and $\beta$,   then: 
\\  $A$ and $B$ are local  $  \iff   \exists \varepsilon \in \ZZZ_{2}, \ \exists N \in \NNN $ such that $\forall c, d \in H$:  \\
$   (z-w)^{N}(A(z)B(w)c,d) = (-1)^{\varepsilon}   (z-w)^{N}(B(w)A(z)c,d)   $ as formal  series.
\end{lemma} 
\begin{proof}  
$( \Rightarrow )$ True by definition. \\
$( \Leftarrow )$ Let $c \in H_{p}$,  $ d \in H_{q}$ \\
$A(n)c \in H_{p-n+\alpha -1} = {0} $ for $n > p+ \alpha -1 $, \\   $B(m)c \in H_{p-m+\beta -1} = {0} $ for $m > p+ \beta -1 $, \\
$A(n)B(m)c$,  $B(m)A(n)c  \in H_{p-(m+n)+\alpha +  \beta -2}$, $ d \in H_{q}$ and $H_{r}  \perp H_{q} $ if $q \neq r$. \\ \\
Let $ S = \{ (m,n) \in \ZZZ^{2};   m+n = p-q + \alpha + \beta -2,  m \le p+ \beta -1 \}  $ \\  and   $ S' = \{ (m,n) \in \ZZZ^{2};   m+n = p-q + \alpha + \beta -2,  n \le p+ \alpha -1 \}  $ \\
 \begin{displaymath} 
(z-w)^{N}(A(z)B(w)c, d) = \sum_{S, k=0}^{N}C_{N}^{k}(A(n)B(m)c,d)z^{-n-1-k}w^{-m-1+N-k}   \end{displaymath}
 \begin{displaymath} 
(z-w)^{N}(B(w)A(z)c, d) =(-1)^{\varepsilon} \sum_{S', k=0}^{N}C_{N}^{k}(B(m)A(n)c,d)z^{-n-1-k}w^{-m-1+N-k}
  \end{displaymath}
But, $S \cap S' $ is a finite subset of $\ZZZ^{2}$, so the formal series is a polynom: $P(A,B,c,d) \in \CCC[z^{\pm 1}, w^{\pm 1}]$; 
now,  using remark \ref{trivial},  and  the fact that \\  $A(n)c = 0 $ for $n > p+ \alpha -1 $  and    $B(m)c  = 0 $ for $m > p+ \beta -1 $, then:
\begin{displaymath}
       (z,w)^{-N}P(A,B,c,d)(z,w) = \left\{  \begin{array}{c} (A(z) B(w)c , d )  \textrm{  \  if  }  \vert z \vert >  \vert w \vert    \\   (-1)^{\varepsilon}(B(w)A(z)c , d )  \textrm{  \  if  }  \vert w \vert >  \vert z \vert     \end{array}  \right. 
\end{displaymath} \end{proof}

 \begin{remark}  (associativity)  \   $(A_{n}B)_{m}C = A_{n}( B_{m}C) =A_{n}B_{m}C   $  \end{remark}   
    
   \begin{lemma}   \label{4}   Let $A_{1} $,..., $A_{R} $ graded,  $A_{i} $ and $A_{j} $ local with $N = N_{ij} \in \NNN$.  Then, $\forall c, d \in H$: 
\begin{displaymath} 
\prod_{i<j}(z_{i} - z_{j} )^{N_{ij}}(A_{1}(z_{1}) ... A_{R}(z_{R}) c , d)  \in \CCC  [ z_{1}^{\pm 1}, ... ,  z_{R}^{\pm 1} ]   
\end{displaymath}       \end{lemma}                                                                                                                                            
\begin{proof}  It is exactly as the previous lemma: \\
We can put each  $A_{i}(z_{i})$ on the first place by commutations. \\ 
We obtain equalities between $R$ series with support $S_{i} \cup T  $, with $T$ the support due to $ \prod_{i<j}(z_{i} - z_{j} )^{N_{ij}}$ (finite), and as the previous lemma: \\ $ S_{i}=  \{ (m_{1},...,m_{R}) \in \ZZZ^{R};   m_{1}+...+m_{R} = K,  m_{i} \le k_{i} \}  $ \\
So, $ \bigcap  S_{i}$   is a finite subset of $ \ZZZ^{R}$, and the result follows. \end{proof} 
\begin{lemma} (Dong's lemma) \ 
Let $A$, $B$, $C$ graded and pairwise local, then $A_{n}B$ and $C$ are local.
  \end{lemma}
\begin{proof}
Let $Q(z_{1}, z_{2}, z_{3}) =   \prod_{i<j}(z_{i} - z_{j} )^{N_{ij}}  $, by lemma  \ref{4}, $\forall d, e \in H$: 
\begin{displaymath}   Q.(A(z_{1})B( z_{2})C(z_{3})d,e)= Q. (-1)^{\varepsilon_{1}+\varepsilon_{2}}(C(z_{3})A(z_{1})B( z_{2})d,e)  \in \CCC  [ z_{1}^{\pm 1},  z_{2}^{\pm 1} ,  z_{3}^{\pm 1} ]  \end{displaymath}
Now, we divide this polynom by $Q$, we fix  $z_{2}$ and we develop around $z_{1}=z_{2}$.  \\
Then $\exists N \in \NNN $ such that  $\forall n \in \ZZZ$ if  $P_{n} $ is  the coefficient of $(z_{1} - z_{2})^{-n-1}$ then $S_{n}= (z_{2} - z_{3})^{N} P_{n} \in \CCC  [ z_{2}^{\pm 1} ,  z_{3}^{\pm 1} ]$. \\
Now, on one hand $S_{n} =   (z_{2} - z_{3} )^{N}(A_{n}B(z_{2})C(z_{3})d,e)$ and on the other hand $S_{n} =   (-1)^{\varepsilon}   (z_{2} - z_{3} )^{N}(C(z_{3})A_{n}B(z_{2})d,e)  $, with 
$\varepsilon = \varepsilon_{1}+\varepsilon_{2}$. \\ 
Then, the result follows by lemmas  \ref{2} and \ref{3}.
 \end{proof}
\begin{pcorollary}  \label{pc}  If in addition,  $A$ and $C$ are local with  $\varepsilon_{1} \in \ZZZ_{2}$ , and, $B$ and $C$, local with $\varepsilon_{2} $,
 then, $A_{n}B$ and $C$ are local with  $\varepsilon = \varepsilon_{1}+\varepsilon_{2}$.
  \end{pcorollary}
 
   \begin{lemma} \label{5}  If $A$ and $B$ are local with $\varepsilon \in \ZZZ_{2}$, so is $A'$ and $B$ 
  \end {lemma}
  \begin{proof} $   (z-w)^{N}(A(z)B(w)c,d) = (-1)^{\varepsilon}   (z-w)^{N}(B(w)A(z)c,d)   $ \\
  Then, applying $\frac{d}{dz}$ and the lemma \ref{3}, the result follows. \end{proof}
  
 \begin{definition} (Operator T) \ Let $T \in  End(H)$. \end{definition}
  \begin{lemma}   Let $A$, $B$ local such that $[T,A] = A' $ and  $[T,B] = B' $. \\
  Then, $[T,A_{n}B]=(A_{n}B)' = A'_{n}B+A_{n}B'  $ and [T,A'] = A''
  \end {lemma}
  \begin{proof}   
  $(z-w)^{N}([T,A(z)B(w)]c,d) = (z-w)^{N}((A'(z)B(w)+A(z)B'(w))c,d) $ \\ \\
  $ =(z-w)^{N}\sum_{n \in \ZZZ}(( A'_{n}B+A_{n}B')(w)c,d)(z-w)^{-n-1} \ \ \ $  on one hand \\
  $ = (z-w)^{N} (\frac{d}{dz}+\frac{d}{dw} ) (\sum_{n \in \ZZZ}  A_{n}B(w) (z-w)^{-n-1} c,d) \ \ \ $   on the other hand \\ \\
  $ =  (z-w)^{N} [ (\sum_{n \in \ZZZ} (-n-1) A_{n}B(w) (z-w)^{-n-2} c,d) +  \\ (\sum_{n \in \ZZZ}  (A_{n}B)'(w) (z-w)^{-n-1} c,d) + (\sum_{n \in \ZZZ}  (n+1)A_{n}B(w) (z-w)^{-n-2} c,d) ]$  \\
  $ = (z-w)^{N} \sum_{n \in \ZZZ} ( (A_{n}B)'(w) c,d)(z-w)^{-n-1}$ \\ \\ 
  By identification:  $[T,A_{n}B]=(A_{n}B)' = A'_{n}B+A_{n}B'  $ \\
  Now, $[T,A] = A'   \Rightarrow [T,A(n)] =-nA(n-1) $, so  $[T,A'] = A''$  \end{proof}

 \begin{lemma}    \label{6}
Let $\Omega \in H$;    $A$, $B$ local with    $A(m)\Omega =  B(m)\Omega = 0 \   \forall m \in \NNN$, \\ 
then $ A'(m)\Omega = A_{n}B(m)\Omega = 0 \  \forall m \in \NNN,   \forall n \in \ZZZ $.
\end {lemma}
 \begin{proof}
$A'(m) = -mA(m-1) $, so $A'(m)\Omega = 0  \ \forall m \in \NNN$    \\
On the formula \ref{formula} ,  $A(n-p)\Omega = B(m+p)\Omega = A(p)\Omega = 0$ because \\ $n-p, m+p, p \in \NNN $,
then, $A_{n}B(m)\Omega = 0 \  \forall m \in \NNN,   \forall n \in \ZZZ $. \end{proof}

\subsection{System of generators}   
\begin{definition} \label{def}  Let $H$ prehilbert space;  $ \{A_{1}$,..., $ A_{r} \} \subset (EndH)[[z,z^{-1}]]$ \\ 
 is a system of generators if $ \exists D, T \in End(H) $, $ \Omega \in H $ such that:
 \begin{description}
\item[(a)]   $\forall i,j \  A_{i}$ and $A_{j}$ are local with $N=N_{ij}$ and $\varepsilon= \varepsilon_{ij}=\varepsilon_{ii} . \varepsilon_{jj} $
\item[(b)]   $\forall i \ [ T , A_{i}  ] = A'_{i} $
\item[(c)]  $D$ decomposes $H = \bigoplus_{ n  \in \NNN + \frac{1}{2} }  H_{n}$  with $D \xi = n \xi $  $\forall \xi \in H_{n}$,  $dim( H_{n} )< \infty $, \\  $ H_{n} \perp H_{m} $   if $n \neq m$ and   
$\forall i  \  A_{i}$  is graded with $\alpha_{i} \in \NNN+\frac{\varepsilon_{ii}}{2} $
\item[(d)]  $ \Omega \in H_{0} $,  $\Vert \Omega \Vert = 1$, and $\forall i \ \forall m \in \NNN $, $
A_{i}(m) \Omega = D \Omega = T\Omega= 0$
\item[(e)]  $\A = \{ A_{i}(m), \forall i \ \forall m \in \ZZZ \}$ acts irreducibly on $H$, so that $H$ is \\  the minimal space containing $\Omega$ and stable by the action of $\A $
\end{description}
\end{definition}
\begin{definition} Let  $S \subset (End H)[[z,z^{-1}]] $, the minimal subset containing $Id$, $A_{1}$,..., $A_{r}$, stable by the operations:
\begin{displaymath} 
  ( A, B) \mapsto (A_{n}B)  \ (\forall n \in \ZZZ ) \quad \quad Ê, \quad  \quad A  \mapsto A'
\end{displaymath}
Let $ S_{\varepsilon} = \{ A \in S \ \vert \    A \  \textrm{is local with itself with} \ \varepsilon \in \ZZZ_{2} \}  $, so that $S = S_{\bar{0}} \amalg  S_{\bar{1}} $. \\
Let $\S_{\varepsilon} = lin< S_{\varepsilon} >$ and $ \S =  \S_{\bar{0}} \oplus  \S_{\bar{1}}$.  \end{definition}
\begin{remark}   All is well defined by previous lemmas.  \end{remark}
 \begin{lemma} $
  \forall A, B \in \S$,   they are local, $A_{n}B \in \S  \ \textrm{and}  \  [T,A] = A' \in \S $
\end{lemma}  \begin{proof} By previous lemmas and  linearizing Dong's lemma.
  \end{proof}
\begin{lemma}   \label{7}  Let $E \in \S_{\varepsilon_{1}}$ and $F \in \S_{\varepsilon_{2}} $ then: 
\begin{description}
\item[(a)]   $E_{n}F \in \S_{\varepsilon_{1} + \varepsilon_{2}}$
\item[(b)]  $E$ and $F$ are local with  $\varepsilon = \varepsilon_{1} . \varepsilon_{2}$
\end{description}
\end{lemma}

\begin{proof}
 (a) \       $E$ and $F$ are local with an   $ \varepsilon \in \ZZZ_{2} $.  \\ 
We use the corollary \ref{pc} with $A=E$,  $B=F$, $C=E$,with  $A=E$,  $B=F$, $C=F$ and finally with $A=E$,  $B=F$, $C=E_{n}F$.  
Then we see that $E_{n}F$ is local with itself with $\varepsilon ' = \varepsilon_{1} +\varepsilon +  \varepsilon_{2} + \varepsilon = \varepsilon_{1} + \varepsilon_{2} $, 
so, $E_{n}F \in \S_{\varepsilon_{1} + \varepsilon_{2}}$ \\ 
(b)  By induction:  \\ 
Base case: $\forall i, j \ A_{i} \in \S_{\varepsilon_{ii} } $, $A_{j} \in \S_{\varepsilon_{jj} } $ and are local with $\varepsilon= \varepsilon_{ij}=\varepsilon_{ii} . \varepsilon_{jj} $ by definition \ref{def}.\\
Inductive step:  We suppose the property for $E \in \S_{\varepsilon_{1}}$,  $F \in \S_{\varepsilon_{2}} $ and  $G \in \S_{\varepsilon_{3}} $.
We prove it for $E_{n}F$ and $G$:  \\
$E$ and $G$ are local with $\varepsilon = \varepsilon_{1} . \varepsilon_{3}$ \\ 
$F$ and $G$ are local with $\varepsilon = \varepsilon_{2} . \varepsilon_{3}$ \\
Now,  $E_{n}F \in \S_{\varepsilon_{1} + \varepsilon_{2}}$,   $G \in \S_{\varepsilon_{3}} $ and by corollary \ref{pc} with $A=E$,  $B=F$, $C=G$, 
$E_{n}F$ and $G$ are local with  $\varepsilon = \varepsilon_{1} . \varepsilon_{3} + \varepsilon_{2} . \varepsilon_{3} = (\varepsilon_{1} + \varepsilon_{2}).\varepsilon_{3} $ \\
The following lemma completes the proof. \end{proof}

\begin{lemma}
$A \in \S_{\varepsilon} \Rightarrow  A' \in \S_{\varepsilon}$
\end{lemma}  
\begin{proof}  By lemma \ref{5},  if $A$ and $B$ are local with $\varepsilon \in \ZZZ_{2}$, so is $A'$ and $B$. \\   The result follows by taking $B= A$ and then $B=A'$.\end{proof}
 
 \begin{definition}   (well defined by lemma \ref{6})  \label{R}
 \begin{displaymath}  \begin{array}{cccc}
    
R:  & \S  &  \longrightarrow  &  H     \\
    &    A  &    \longmapsto  &  a:= A(z)\Omega_{\vert z=0}     \end{array}      \quad  \mathrm{linear.}    \end{displaymath}
\end{definition} \newpage

\begin{examples}        $ \begin{array}{c}    \end{array} $ 
\begin{description} 
\item[(a)]  $R(Id) = \Omega $, \   $R(A) = A(-1)  \Omega $
\item[(b)] $R(A') = A(-2) \Omega = T. R(A)$
 \item[(c)]  $R(A_{n}B) = A(n)R(B) $ (by formula \ref{formula}) 
 \item[(d)]  $R(A_{n}Id) = A(n) \Omega $
 \end{description}
\end{examples}  

\begin{lemma}  \label{8}
$A$ is  graded with $\alpha$ $\iff$  $R(A) \in H_{\alpha}$ \end{lemma}
\begin{proof}  By lemma \ref{1} and \ref{2}, inductions and linear combinations.  \end{proof}

\paragraph{State-Field correspondence:}
\begin{lemma} (Existence)  
 $\forall a \in H$, Ê$\exists A \in \S $ such that  $ R(A) = a$.
\end{lemma}
\begin{proof}
$R((A_{i_{1}})_{m_{1}} (A_{i_{2}})_{m_{2}} ... (A_{i_{k}})_{m_{k}} Id) = A_{i_{1}}( m_{1}) R((A_{i_{2}})_{m_{2}} ... (A_{i_{k}})_{m_{k}} Id) \\  =...= A_{i_{1}}( m_{1}) ...A_{i_{k}} ( m_{k}) \Omega $ \\
Now, the action of the $A_{i}( m)$ on $\Omega$ generates $H$ by definition \ref{def}. \end{proof}
\begin{lemma}
Let $A \in \S$, then $A(z) \Omega = e^{zT}R(A) $.
\end{lemma}
\begin{proof}
Let $F_{A}(z) = A(z) \Omega = \sum_{n \in \NNN}A(-n-1) \Omega z^{n} $, \\
Then, $\forall b \in H$, $(F_{A}(z),b) \in \CCC[z] $\\ 
Now, $\frac{d}{dz}(F_{A}(z),b) = (\frac{d}{dz}F_{A}(z),b) = (A'(z)\Omega, b) \\  = ([T , A(z)]\Omega, b) = (T.A(z) \Omega, b)= (T.F_{A}(z) \Omega, b)$ \\
But,  $F_{A}(0) = R(A) $, so we see that:  $ (F_{A}(z),b) = (e^{zT}R(A),b) $ $\forall b \in H $ \\
Finally,  $F_{A}(z) = e^{zT}R(A) $ \end{proof}

\begin{lemma} (Unicity) \ \     
 $R(A) = R(B) \Rightarrow A=B$.  
 \end{lemma}
   
 \begin{proof}  Let $C = A-B$, then  $R(C) = R(A) - R(B) = 0$  \\  and $F_{C}(z) = e^{zT}R(C) = 0  $  \\ 
 Now, $\forall e \in H  $, $\exists E \in \S $ such that $R(E) = e$.  \\ Then $\forall f \in H$, $\exists N \in \NNN $ $\exists \varepsilon \in \ZZZ_{2} $  such that :   \\
 $(z-w)^{N} ( C(z)E(w) \Omega , f) = (-1)^{\varepsilon} (z-w)^{N} ( E(w) C(z) \Omega , f)  $  \\
 Now, $ (E(w)C(z) \Omega , f) = (E(w) F_{C}(z), f) = 0 = (C(z)E(w) \Omega , f)$ \\
 So, $(C(z)E(w) \Omega , f)_{\vert w=0} = (C(z)e , f) = 0$  $\forall e, f \in H $ \\
 Finally, $C=0$ and $A=B$
 \end{proof}  
 Now, we can well defined:  
 \begin{definition}  (State-Field correspondence map) 
  \begin{displaymath}   \begin{array}{cccc}
V:  & H &  \longrightarrow  &  \S     \\
    &    a  &    \longmapsto  &  V(a)     \end{array}   \quad   \textrm{  linear.}    \end{displaymath} 
    such that :  
     $  \left(  \begin{array}{cc}   \forall a \in H  &   R(V(a))=a     \\  
                                                      \forall A \in \S    &  V(R(A)) = A   \end{array}  \right. $
                                      \end{definition}

 \begin{notation}
 $V(a)(z)$ is noted $V(a,z)$ and $A(z) = V(R(A), z)$
 \end{notation}
 
 \begin{examples}        $ \begin{array}{c}    \end{array} $ 
\begin{description} 
\item[(a)]  $V(0 , z) = 0  $, \   $V(\Omega ,z) = Id $
\item[(b)] $V'(a,z) = V(T.a, z)$
 \item[(c)]  $(A_{n}B)(z) = V(A(n)R(B),z) $ 
 \end{description}
\end{examples}  

\begin{definition} Let  $ H_{\varepsilon } = \bigoplus_{ n \in  \NNN + \frac{\varepsilon }{2}} H_{n} $ so that   $ H =  H_{\bar{0}} \oplus H_{\bar{1}}$. \end{definition} 

\begin{lemma} $R(\S_{\varepsilon}) = H_{\varepsilon }  \quad   ( \varepsilon \in  \ZZZ_{2}  ) $
\end{lemma}
\begin{proof} Base step: by definition \ref{def} and lemma \ref{8},  \\   $\forall i$   $A_{i}  \in \S_{\varepsilon_{ii}} $ and $R(A_{i}) \in H_{\alpha_{i}}$ with $\alpha_{i} \in \NNN + \frac{\varepsilon_{ii}}{2}$ \\
Inductive step:  by lemma \ref{7}  \end{proof}

\begin{corollary} (Relation with $T$ and $D$)  \  Let $ a \in H_{\alpha}$,  we have that:  
 \begin{description}
 \item[(a)] $[T, V(a,z) ] = V'(a,z) = V(T.a , z)  \in \S$ 
 \item[(b)]  $[D, V(a,z) ] = z.V'(a,z)  + \alpha.V(a,z)  \quad $  ( $\notin \S$   in general)
\end{description}   \end{corollary}

\subsection{Application to fermion algebra}  \label{sec2}

 $ H = \F_{NS}$, $\psi (z) = \sum_{n \in \ZZZ} \psi_{n+ \1/2} z^{-n-1}$ with $ [  \psi_{m} ,  \psi_{n}  ]_{+} = \delta_{m+n}Id$.

  \begin{proposition}  $ \{ \psi  \} $ is a system of generator.
    \end{proposition} 
     \begin{proof}
    $ \psi $ is local with itself with $N=1$ and $\varepsilon  =  \bar{1}=  \bar{1} . \bar{1} $ (see definition \ref{def}) \\
    We have construct $D$ and $T$  (p \pageref{D} ), $\Omega \in H_{0}$,  $\Vert  \Omega   \Vert = 1 $,  $D\Omega = T\Omega = 0 $. \\
    $[T, \psi (z) ] = \psi ' (z) $,  $[D, \psi (z)] = z.\psi ' (z) + \1/2 \psi (z) $  and   $\1/2 \in \NNN + \1/2 $\\
    Finally, $\{   \psi_{n} , n \in \1/2 \NNN \} $ acts irreducibly on $H$    \end{proof}
\begin{corollary}  $ \{ \psi  \} $ generates an $\S$ with a state-field correspondence with: 
 \begin{center}  $ R(\psi) = \psi_{-\1/2} \Omega$ and  $ \psi (z)  = V(\psi_{-\1/2} \Omega, z) $  \end{center}  \end{corollary}
 
\begin{lemma}  (OPE) \ \   $  \psi (z)\psi (w)  \sim     \frac{Id}{z-w}     $    \end{lemma}
\begin{proof}
  $  \quad   \psi_{n}\psi (w) = V(\psi_{n+\1/2}  \psi_{-\1/2} \Omega, w) = 0$  if $n \ge 1 $  ( here $N = 1$ )  \\
 Now, for  $0 \le n \le N-1 $ i.e $n=0$ :  \\
  $\psi_{\1/2}\psi_{-\1/2}  \Omega  =  ( [  \psi_{\1/2} ,  \psi_{-\1/2}  ]_{+} - \psi_{-\1/2}  \psi_{\1/2} ) \Omega = \Omega$,  so \   $\psi_{0}\psi (w) = Id  \  $ \end{proof}
   
\begin{remark}  (Next operator)  $ \psi_{-\1/2} \psi_{-\1/2}  \Omega = 0 $, so $ \psi_{-1 }\psi  =  0 $; and the next operator of the expansion is 
$2 L(w) :=   \psi_{-2}\psi (w)= 2 \sum_{n \in \ZZZ} L_{n} z^{-n-2}$ \\ 
   Now,  $R(L) = \1/2 \psi_{-\3/2}\psi_{-\1/2}  \Omega$,   then $L(w)=   V(\1/2 \psi_{-\3/2}\psi_{-\1/2}  \Omega , w) $.   \end{remark}
 \begin{remark} $L(n) = L_{n-1} $ so, $L_{0}\Omega = L_{-1}\Omega = 0 $   by lemma \ref{6}.   \end{remark} 
 \begin{lemma} (OPE) \ \  $   \label{s1}  \psi (z) L(w)  \sim     \frac{1/2  \psi (w) }{(z-w)^{2}} -   \frac{1/2  \psi '  (w) }{(z-w)}    $   \end{lemma}
 \begin{proof} $    \psi_{n}L(w) = \1/2 V(\psi_{n+\1/2}  \psi_{-\3/2}  \psi_{-\1/2} \Omega, w) = 0$  if $n \ge 2 $  ( here $N = 2$ )  \\
Now,   $  \psi_{\1/2}  \psi_{-\3/2}  \psi_{-\1/2} \Omega =  - \psi_{-\3/2} \Omega = R(\psi ')$ \  ,  \     $ \psi_{\3/2}  \psi_{-\3/2}  \psi_{-\1/2} \Omega =   \psi_{-\1/2} \Omega = R(\psi ')$ 
\end{proof}

\begin{lemma} (Lie bracket)  \ \  $ [L_{m} , \psi_{n} ] = - (n + \1/2 m) \psi_{m+n} $ \end{lemma}
\begin{proof}
By lemma \ref{7}, $\psi$ and $L$ are local with $\varepsilon = \bar{0}$, and by formula \ref{lie}:   \\
$[\psi (m) , L(n+1)  ] = -\1/2 C_{m}^{0} \psi ' (m+n+1)  +  \1/2  C_{m}^{1} \psi  (m+n+1-1)  $   \\
$=  \1/2 (m+n+1) \psi  (m+n)  +  \1/2  m \psi  (m+n)  = (m + \1/2 + \1/2 n)  \psi  (m+n) $ \\ 
We have computed   for $m \ge 0$, we find the same result for $m < 0$. \\
 Now,  $\psi (m) = \psi_{m+\1/2} $ and $L(n+1) = L_{n} $, so the result follows.   \end{proof}
 \begin{lemma}  $D = L_{0}$  and  $T = L_{-1} $  \end{lemma}
 \begin{proof}
$ [L_{0} , \psi_{n} ] = - n \psi_{n} = [D ,  \psi_{n} ]$ ,    $ [L_{-1} , \psi_{n} ] = - (n - \1/2 ) \psi_{n-1} = [T ,  \psi_{n} ] $  \\
So, by irreducibility and Schur's lemma,    $L_{0} - D and   L_{-1} - T   \in \CCC Id$ \\
Now,  $L_{0}\Omega =  D\Omega =   L_{-1}\Omega =   T\Omega =0$,   then,  $D = L_{0}$  and  $T = L_{-1}  $  \end{proof}
 \begin{corollary}   $\forall  a \in H_{s}$:  
  \begin{description}
 \item[(a)] $[ L_{-1}, V(a,z) ] = V'(a,z) = V( L_{-1}.a , z)  \in \S$ 
 \item[(b)]  $[L_{0}, V(a,z) ] = z.V'(a,z)  + s.V(a,z)  $ \end{description}   \end{corollary}
 
 \begin{remark}  $\forall A \in \S$, $A' = (L_{0}A)$,  so, by Dong's lemma, we finally don't need here to $A \mapsto A' $ for the construction of $\S$.      \end{remark} 
 
 \begin{lemma} (OPE) \ \ $ \label{s2}     L(z) L(w)  \sim     \frac{(c/2) Id }{(z-w)^{4}} +   \frac{2  L(w) }{(z-w)^{2}} +   \frac{ L'(w) }{(z-w)}    $    \end{lemma}
 \begin{proof}
 $  \quad   L_{n}L (w) =  V(L(n)L(-1) \Omega , w) = V(L_{n-1} L_{-2} \Omega , w) = 0 $  if $n \ge 4 $. \\
 Then, here, $N = 4$, so, for $0 \le n \le N-1 $:  
  \begin{description}
 \item[(a)]    $ V(L_{-1} L_{-2} \Omega , w) = L'(w) $
 \item[(b)]    $L_{0} L_{-2} \Omega =  2 L_{-2} \Omega = 2 R(L)  \quad  $    because   $L_{-2} \Omega  \in H_{2} $
 \item[(c)]    $L_{1} L_{-2} \Omega = \1/2 L_{1} \psi_{-\3/2}\psi_{-\1/2}  \Omega =  \1/2 [ L_{1},  \psi_{-\3/2}]\psi_{-\1/2} \Omega = \1/2  \psi_{-\1/2}^{2}\Omega = 0  $
 \item[(d)]    $ L_{2} L_{-2} \Omega  \in H_{0} = \CCC \Omega $,  so,  $L_{2} L_{-2} \Omega = K  \Omega$   with   $K = \Vert L_{-2} \Omega  \Vert ^{2}$  
 \end{description}  \end{proof} \begin{notation} $c:=  2 \Vert L_{-2} \Omega  \Vert ^{2}, $ the central charge.  \\ 
 $ (  \mathrm{here} \   c =  \1/2 (  \psi_{-\3/2}\psi_{-\1/2}  \Omega ,  \psi_{-\3/2}\psi_{-\1/2}  \Omega ) = \1/2  ) $  \end{notation}
 
  \begin{notation} Let $\delta_{k} = \left\{  \begin{array}{c}     0   \quad  \mathrm{ if  } \quad    k \neq 0   \\   Id    \quad  \mathrm{ if  }  \quad   k = 0   \end{array}   \right.   $    \end{notation}
  
 \begin{lemma} (Lie bracket) \ $   [L_{m} , L_{n} ] = (m-n) L_{m+n}  +  \frac{c}{12} m(m^{2}-1) \delta_{m+n}    $. \end{lemma}
 \begin{proof}
By lemma \ref{7},    $L  \in \S_{\bar{0}}$,  and by formula \ref{lie}: \\ 
 If $m+1 \ge 0$, then:    $[ L(m+1) , L(n+1) ] =  \\ 
 C_{m+1}^{0} L'(m+n+2) + 2C_{m+1}^{1} L(m+n+2-1) + \frac{c}{2}  C_{m+1}^{3}  Id(m+n+2-3) \\
  =   -(m+n+2)L(m+n+2) + 2(m+1) L(m+n+1) +  \frac{c}{2} \frac{m(m^{2}-1)}{6} \delta_{m+n} \\
  =   (m-n) L(m+n+1)  +  \frac{c}{12} m(m^{2}-1) \delta_{m+n}$ \\   
    We find the same result for $ m+1 < 0$ \end{proof}
      \begin{remark} \label{star}  $ L_{m}^{\star } = L_{-m}$   \end{remark} 
    \begin{proof}  $[\psi_{-n} ,  L_{m}^{\star }] = [L_{m} ,\psi_{n}  ]^{\star } = - (n + \1/2 m) \psi_{-m-n} = [\psi_{-n} ,  L_{-m} ]   $, 
   then the result follows by irreducibility, Schur's lemma and grading. \end{proof}     \newpage 
    \begin{remark} The ($L_{n}$) generate a Virasoro algebra $\Vir$.  \end{remark} 
 \begin{corollary} $\Vir$  acts on $H =  \F_{NS} $, and admits $L(c,h) = L( \1/2 , 0)$ \\  as minimal submodule containing $ \Omega$. \end{corollary}
 \begin{definition}     Let  call $L$ the Virasoro operator, \\  and $\omega =  R(L) = \1/2 \psi_{-\3/2}\psi_{-\1/2}  \Omega $, the Virasoro vector.  \end{definition}     
   
   \subsection{Vertex operator superalgebra}  
  \begin{definition} \label{vertexdef}  A vertex operator superalgebra is an $( H, V,  \Omega, \omega) $  with: 
 \begin{description}
\item[(a)] $H = H_{\bar{0}} \oplus  H_{\bar{1}} $ a prehilbert superspace.
\item[(b)] $ V:    H    \rightarrow      (End H)[[z,z^{-1}]] $ a  linear map.
 
\item[(c)] $ \Omega, \hspace{0,065cm} \omega \in H$  the  vacuum and Virasoro vectors.
 \end{description} Let $\S_{\varepsilon} = V (H_{\varepsilon} )$,  $\S = \S_{\bar{0}} \oplus \S_{\bar{1}}$ and  $  A(z) = V(a ,z) = \sum_{n \in \ZZZ}A(n)z^{-n-1} $,  \\
 then $( H, V, \Omega, \omega) $ satisfies the followings axioms: 
\begin{enumerate}
\item  (vacuum axioms):   $\forall A \in \S$ and $\forall n \in \NNN$, $A(n) \Omega = 0$, \\  $V(a,z)\Omega_{\vert z=0} = a $ and $V(\Omega , z) = Id $  
\item   (irreducibility axiom):   Let $\A = \{ A(n) \vert  A \in \S ,  n \in  \ZZZ \}$ then, \\  $\A$ acts irreducibly on $H$, so that $\A . \Omega = H$
\item   (locality axiom):   $\forall A \in \S_{\varepsilon_{1}}$, $\forall B \in \S_{\varepsilon_{2}}$,  $A$ and $B$ are local \\ 
(see definition \ref{local} and lemma \ref{3}),  with $\varepsilon = \varepsilon_{1} . \varepsilon_{2} $   and $A_{n}B \in \S_{\varepsilon_{1} +  \varepsilon_{2}}$  	
\item (Virasoro axiom):  $V(\omega , z) = L(z) = \sum_{n \in \ZZZ}L_{n} z^{-n-2} $ Virasoro operator  $ (L_{0} \Omega = L_{-1} \Omega =  0$ and $   \omega = L_{-2}\Omega ) $.  Let  $c=  2 \Vert  \omega \Vert^{2}$  the central charge: \\ 
$  [L_{m} , L_{n} ] = (m-n) L_{m+n}  +  \frac{c}{12} m(m^{2}-1) \delta_{m+n}   $ 
\item  ($L_{0}$ axioms)  $L_{0}$ decomposes $H$ into $\bigoplus_{ n  \in \NNN + \frac{1}{2} } H_{n} $ with $dim( H_{n} )< \infty  $,   $ H_{n} \perp H_{m} $   if $n \neq m$,    $ H_{\varepsilon } = \bigoplus_{ n \in  \NNN + \frac{\varepsilon }{2}} H_{n} $, $\Omega \in H_{0}$, $\omega \in H_{2}$, and \\  
 $\forall a \in H_{\alpha}$,  $[L_{0}, V(a,z) ] = z.V'(a,z)  + \alpha.V(a,z)  $  
 \item  ($L_{-1}$ axioms):     $[ L_{-1}, V(a,z) ] = V'(a,z) = V( L_{-1}.a , z)      \in \S   $
\end{enumerate}
 \end{definition}    \newpage
 
 \begin{corollary} A system of generators, generating a Virasoro operator $L \in \S $, with $D = L_{0}$ and $T = L_{-1}$, generates a vertex operator superalgebra.
 \end{corollary}
 
 \begin{corollary}
 The fermion operator $\psi$ generates a vertex operator superalgebra , with Virasoro vector $\omega =  \1/2 \psi_{-\3/2}\psi_{-\1/2} \Omega$. 
 \end{corollary}
  \begin{remark}  The Virasoro operator $L$ alone, generates the minimal vertex operator (super)algebra.
  \end{remark}
  \begin{remark} 
Let $A(z) = V(a,z)$ and  $B(w) = V(b, w) $; the formula \ref{formula} is general, so similary, by vacuum axioms,  $A_{n}B(w) = V(A(n)b , w) $. 
   \end{remark}
 \begin{proposition} (Borcherds associativity)  $ \exists N \in \NNN $ such that $\forall c, d \in H$:  
 \\Ê $(z-w)^{N}(V(a,z)V(b,w)c,d) =  (z-w)^{N} (V(V(a, z-w)b,w)c,d) $
 \end{proposition}
 \begin{proof}
 To simplify the proof, we don't write: \\  "$ \exists N \in \NNN $ such that $\forall c, d \in H$ $(z-w)^{N}( \  . \  c,d) $",   but it is implicit. \\ \\
 $V(a,z)V(b,w) = A(z)B(w) = \sum A_{n}B(w) (z-w)^{-n-1} \\ = \sum V(A(n)b , w) (z-w)^{-n-1} = V(\sum A(n)b (z-w)^{-n-1} , w) \\ = V(\sum A(n)(z-w)^{-n-1}  b , w)  = V(V(a, z-w)b,w) $.
 \end{proof}

\newpage 
\section{Vertex $\gg$-superalgebras and modules}   \label{vertex4}
\subsection{Preliminaries }
\subsubsection{Simple  Lie algebra $\gg$}   Let $\gg$ be a simple Lie algebra of dimension $N$,  a basis $( X_{a} )$  \\Êwith 
$ [X_{a} , X_{b} ] = i \sum_{c} \Gamma_{ab}^{c} X_{c}$    with   $\Gamma_{ab}^{c}  \in \RRR$ totally antisymmetric. 
\begin{lemma}   Let   $\C = \sum_{b} X_{b}^{2}$, then $[\gg, \C ] = 0 $
\end{lemma}   
\begin{proof}
It suffices to prove $[X_{a}, \C ] = 0 $  for each $X_{a}$. \\
 $[X_{a}, \C ] =   \sum_{b} [X_{a} , X_{b}^{2} ]  = \sum_{b} ([X_{a} , X_{b} ]X_{b}  + X_{b}[X_{a} , X_{b} ]) =   i \sum_{b,c} \Gamma_{ab}^{c} X_{c}X_{b}  +$ \\ 
$ i \sum_{b,c} \Gamma_{ab}^{c} X_{b}X_{c} =  i \sum_{b,c}( \Gamma_{ab}^{c} + \Gamma_{ac}^{b} ) X_{c}X_{b}  = 0 $ \quad by  antisymmetry.  \end{proof}
\begin{remark}  \label{casimir} $\C$ is a multiple of the \textbf{Casimir} of $\gg$. We suppose to have well normalized the basis such that $\C$ is exactly the Casimir. 
\end{remark}
\begin{corollary} \label{cas}
By Schur's lemma, $\C $ acts as multiplicative constant $c_{V}$ on each irreducible representation $V$.
\end{corollary}
\begin{example} $\gg$ is simple, it  acts irreducibly on  $V = \gg$ with $ad$.  \end{example}
\begin{lemma}   $\sum_{a,c} \Gamma_{ac}^{b} . \Gamma_{ac}^{d} = \delta_{bd} c_{\gg}  $  
\end{lemma}
 \begin{proof}  
 $ (\sum_{a}ad_{X_{a}}^{2})(X_{b}) = c_{\gg} X_{b} = \sum_{a}[X_{a}, [X_{a} , X_{b} ] ] \\ =  i^{2} \sum_{a,c,d} \Gamma_{ab}^{c} . \Gamma_{ac}^{d} X_{d}=\sum_{a,c,d} \Gamma_{ac}^{b} . \Gamma_{ac}^{d} X_{d} $. \\
Then,  $\sum_{a,c} \Gamma_{ac}^{b} . \Gamma_{ac}^{d}  = \delta_{bd} c_{\gg}  $ 
 \end{proof}
\begin{definition}  $ g= \frac{c_{\gg}}{2} $  is called the dual Coxeter number.

\begin{example}  $\gg = A_{1} =  \sl_{2}$, $dim(\gg) = 3$ \\                                              
 $[E,F] = H $, $[H,E] = 2E $, $[H,F] = -2F $, with Casimir $EF+FE+\1/2H^{2}$ \\      
We choose the  basis:  $X_{1}= \frac{i\sqrt{2}}{2} (E-F)$,  $X_{2}=  \frac{\sqrt{2}}{2}(E+F)$, $X_{3}= \frac{\sqrt{2}}{2}H$, \\   
with relations:  $[X_{1},X_{2}]= i\sqrt{2} X_{3}$, $[X_{3},X_{1}]= i\sqrt{2} X_{2}$,  $[X_{2},X_{3}]= i\sqrt{2}X_{1}$ \\    
$\C = \sum_{a} X_{a}^{2} = EF+FE+\1/2H^{2}$  and      $g= \1/2\sum_{a,b} (\Gamma_{ab}^{c}  )^{2} = 2$
\end{example}
\begin{center} Table (see \cite{8} p 111)   \end{center}  
\begin{tabular}{|c|c|c|c|c|c|c|c|c|c|}
\hline
  $\gg$ & $A_{n}$ & $B_{n}$  & $C_{n}$  & $D_{n}$  & $ E_{6}$ & $E_{7}$  & $E_{8}$ & $F_{4}$ & $G_{2}$    \\
  \hline
  $dim(\gg)$ & $n^{2} +  2n$ &$2n^{2} + n $&$2n^{2} + n $&$2n^{2} - n$ &$78$ & $133$&$ 248$ &$52$ &$14 $ \\
  \hline
   $g$ & $n+1$  & $2n-1$  & $n+1$  & $2n-2$  &$12$ & $18$&$30$ &$9$ & $4$ \\
\hline
\end{tabular} \end{definition}

 \newpage 
\subsubsection{Loop algebra $L \gg$ }    \label{loop}
\begin{definition} Let  $L \gg = C^{\infty}(\SSS^{1} , \gg)$ the loop algebra of $\gg$. \\  It's an infinite dimensional Lie $\star$-algebra,  admitting  the 
$X_{n}^{a} = X_{a} e^{in \theta}$ \\  as basis, with $n \in \ZZZ$ and  $(X_{a}) $ the base of $\gg$;  so:
\begin{displaymath} [X_{m}^{a} , X_{n}^{b}] = [ X_{a} , X_{b}]_{m+n}    \quad  \mathrm{and} \quad      (X_{n}^{a})^{\star} = X_{-n}^{a}
\end{displaymath}   \end{definition}
\begin{proposition} (Boson cocycle)  $L \gg $ has a unique central extension, up to equivalent, i.e. $H_{2}(L \gg , \CCC)$ is $1$-dimensional.
  $H_{2}(L \gg , \CCC)$ is $1$-dimensional. Let $\L$ the central element and $\widehat{\gg} _{+} = L \gg \oplus \CCC \L$ called $\gg$-boson algebra, then: 
\begin{displaymath} [X_{m}^{a} , X_{n}^{b}] = [ X_{a} , X_{b}]_{m+n} + m\delta_{ab}\delta_{m+n}.\L  \end{displaymath}
\end{proposition}  \begin{proof} See \cite{prse} or   \cite{1} p 46. \end{proof}
\begin{theorem}
The unitary highest weight representations   of  $\widehat{\gg} _{+}$ are \\ $H=L(V_{\lambda} , \ell)$ with: 
\begin{description}
\item[(a)]      $ \ell \in \NNN$ such that $\L \Omega  = \ell \Omega $  \ \ (the level of $H$).
\item[(b)]     $H_{0} = V_{\lambda}$ irreducible representation of $\gg$ such that: \\
$(\lambda , \theta) \le \ell$   with $\lambda$ the highest weight and $ \theta$ the highest root.  
\end{description}    \end{theorem}
\begin{proof} See \cite{prse} or \cite{1} p 48. 
   \end{proof}
 
 \begin{remark}
Let $\CCCC_{\ell} $  the category of such representations for $\ell$ fixed. \\ $\CCCC_{\ell} $ is a finite set and  $\CCCC_{\ell} \subset  \CCCC_{\ell+1}$
\end{remark}
  \begin{remark}  The irreducible unitary projective positive energy representations of $L \gg $ are  given by the unitary highest weight representation of $\widehat{\gg} _{+}$. \end{remark}
\begin{example} \label{ex}   We take $\gg = \sl_{2} $, then $H = L(j , \ell) $ with: 
\begin{itemize}
\item       $\L \Omega  = \ell \Omega $,   \quad $\ell \in \NNN$
\item     $H_{0} = V_{j}$ with $j \in \1/2\NNN$ the spin and   $j \le \frac{\ell}{2}$, such that \\  $\C \Omega = c_{V_{j}} \Omega$
 with $\C = \sum_{a} (X_{0}^{a})^{2}$ the Casimir and $c_{V_{j}} = 2j^{2}+2j$     \end{itemize}
\end{example}  \newpage

\subsection{$\gg$-vertex operator superalgebras}   \label{546}
\subsubsection{$\gg$-fermion} \label{fer}
\begin{definition}  Let $\widehat{\gg}_{-}$ be the $\gg$-fermion algebra,  generated by $(\psi^{a}_{m})$ with  $a \in \{1,...,N\} $, $ N= dim(\gg) $, 
   $m \in \ZZZ+\1/2 $ and relations: 
    \begin{displaymath} [\psi^{a}_{m},\psi^{b}_{n}  ]_{+} = \delta_{ab} \delta_{m+n}  \quad \textrm{and }   \quad  (\psi^{a}_{m})^{\star} =  \psi^{a}_{-m}    \end{displaymath} \end{definition}
\begin{remark} As for the fermion algebra of section \ref{sec1}, we generate the Verma module   $H =  \F_{NS}^{\gg}$, and  the sesquilinear form $( . Ê, . )$  which is a scalar product;   $\pi(\psi_{n}^{a})^{\star} =  \pi((\psi_{n}^{a})^{\star})$,     $\F_{NS}^{\gg}$ is a prehilbert space, an irreducible representation of $\widehat{\gg}_{-}$ and its unique unitary highest weight representation.   \end{remark}
\begin{definition} Let   $\psi^{a}(z) = \sum_{n \in \ZZZ} \psi^{a}_{n+\frac{1}{2}}.z^ {-n-1}$  the fermion operators. \end{definition}
\begin{remark}    $\psi^{a}(z) \psi^{b}(w) \sim \frac{\delta_{ab}}{(z-w)}  $  \end{remark}
\begin{remark}As for the single fermion operator $ \psi  $, of section \ref{sec2}, \\  $\{ \psi^{a}, a \in \{1,...,N\}  \}$ generates a vertex  operator superalgebra with:
\begin{displaymath}
  \omega = \1/2\sum_{a} \psi_{-\3/2}^{a}\psi_{-\1/2}^{a}\Omega  \quad \mathrm{and} \quad  c=  2 \Vert  \omega \Vert^{2} = \frac{dim(\gg)}{2}
\end{displaymath}    \end{remark}
\begin{definition} \label{S} Let $S^{c}(z) = V(s^{c},z) = \sum_{n \in \ZZZ}S_{n}^{c}z^{-n-1}$ with: 
\begin{center} $s^{c} = -\frac{i}{2}\sum_{a,b}\Gamma_{ab}^{c}\psi_{-\1/2}^{a}\psi_{-\1/2}^{b} \Omega \in H_{1} \subset H_{\bar{0}} $ \end{center}
  \end{definition}
\begin{lemma} \ (OPE and Lie bracket) \label{actiono}
 \begin{displaymath}
\psi^{a}(z)S^{b}(w)    \sim    \frac{ i\sum_{c} \Gamma_{ab}^{c} \psi^{c}(w)}{(z-w)}   \quad    \mathrm{and}   \quad  [\psi^{a}_{m} ,  S^{b}_{n}] = i\sum_{c} \Gamma_{ab}^{c}  \psi^{c}_{m+n}  =  [ S^{a}_{m}, \psi^{b}_{n}]  \end{displaymath}  \end{lemma}
 \begin{proof}   $\psi_{n+\1/2}^{d}. s^{c} =0$  if  $n \ge 1 $ and $\psi_{\1/2}^{d}. s^{c} = i\sum_{a} \Gamma_{dc}^{a}\psi_{-\1/2}^{a} \Omega$.   \end{proof}
  \begin{remark}       $ [ S^{a}_{m}, \psi^{a}_{n}] = 0 $ \end{remark}
 \begin{lemma}$(S^{b}_{m})^{\star} =  S^{b}_{-m}$   \label{etoile}  \end{lemma}
\begin{proof}  $ [(S^{b}_{n})^{\star},\psi^{a}_{-m}  ] = [\psi^{a}_{m}, S^{b}_{n}  ]^{\star}  = -i\sum_{c}\Gamma_{ab}^{c}\psi^{c}_{-m-n}  
= [S^{b}_{-n},\psi^{a}_{-m}  ] $  \\
 The result follows by irreducibility, Schur's lemma and grading.
\end{proof}
\begin{remark} (Jacobi)  $[X_{a},[X_{b},X_{c}]] = [[X_{a},X_{b}],X_{c}] + [X_{b},[X_{a},X_{c}]]$ \\
$\Leftrightarrow$  $\sum_{d}\Gamma_{bc}^{d}\Gamma_{ad}^{e} = \sum_{d} (\Gamma_{ab}^{d}\Gamma_{dc}^{e}+\Gamma_{ac}^{d}\Gamma_{bd}^{e} ) $  $\Leftrightarrow$
 $\sum_{e} ( \Gamma_{ab}^{e}\Gamma_{cd}^{e}+\Gamma_{da}^{e}\Gamma_{cd}^{e}+\Gamma_{db}^{e}\Gamma_{ac}^{e} )= 0 $ 
\end{remark}
\begin{notation}  $[S^{a}, S^{b}]:= i \sum_{c} \Gamma_{ab}^{c}S^{c}$
\end{notation}

\begin{lemma}  \ (OPE and Lie bracket) 
  \begin{displaymath}  S^{a}(z)S^{b}(w)   \sim   \frac{[S^{a}, S^{b}](w)}{(z-w)} + \frac{ g. \delta_{ab}}{(z-w)^{2}}     \end{displaymath}
\begin{center}and \  $  [S_{m}^{a} , S_{n}^{b}] = [ S^{a} , S^{b}](m+n) + \ell . m \delta_{ab}\delta_{m+n}$   \quad \  (with  \  $\ell = g \in \NNN $) \end{center}   \end{lemma}
\begin{proof}
 $S_{n}^{d}s^{c} = 0$  if  $n \ge 2 $ \   and:  
\begin{description}
\item[(a)]  $S_{0}^{d}s^{c} =-\frac{i}{2}\sum_{a,b}\Gamma_{ab}^{c}S_{0}^{d}\psi_{-\1/2}^{a}\psi_{-\1/2}^{b} \Omega \\=  -\frac{i}{2}(i\sum_{a,b,e}\Gamma_{ab}^{c}\Gamma_{da}^{e}\psi_{-\1/2}^{e}\psi_{-\1/2}^{b} \Omega+ i\sum_{a,b,e}\Gamma_{ab}^{c}\Gamma_{db}^{e}\psi_{-\1/2}^{a}\psi_{-\1/2}^{e} \Omega  ) \\Ê=  -\frac{i}{2}(i\sum_{a,b,e}(\Gamma_{eb}^{c}\Gamma_{de}^{a}+\Gamma_{ae}^{c}\Gamma_{de}^{b})\psi_{-\1/2}^{a}\psi_{-\1/2}^{b} \Omega \\= i \sum_{e} \Gamma_{dc}^{e}\frac{-i}{2} \sum_{a,b}\Gamma_{ab}^{e}\psi_{-\1/2}^{a}\psi_{-\1/2}^{b}\Omega 
=  i \sum_{e}\Gamma_{dc}^{e}s^{e} = [S^{d} , S^{c}](-1) $  
\item[(b)]  $S_{1}^{d}s^{c} =-\frac{i}{2}i \sum_{a,b,e}\Gamma_{ab}^{c}\Gamma_{da}^{e}\psi_{\1/2}^{e}\psi_{-\1/2}^{b}\Omega  = \frac{1}{2} \sum_{a,b}\Gamma_{ab}^{c}\Gamma_{ab}^{d} = g. \delta_{cd} $
\end{description}   \end{proof}
\begin{corollary} \label{bosonaction} $(S_{m}^{a}) $ is the basis of a $\gg$-boson algebra.
\\ It admits $L(V_{0}, g)$ as minimal submodule of $\F_{NS}^{\gg} $ containing  $\Omega$ $ \\ ($with $V_{0} = \CCC $ the trivial representation of $\gg)$. 
\end{corollary}
\begin{lemma}   $\sum_{a}(S_{-1}^{a})^{2}\Omega = 4g \omega$
\end{lemma}
\begin{proof}
$ \sum_{e}(S_{-1}^{e})^{2}\Omega  =   -\frac{i}{2}\sum_{a,b,e}\Gamma_{ab}^{c}S_{-1}^{e}\psi_{-\1/2}^{a}\psi_{-\1/2}^{b} \Omega
 \\=  -\frac{1}{4} \sum_{a,b,c,d,e}\Gamma_{ab}^{e}\Gamma_{cd}^{e} \psi_{-\1/2}^{a}\psi_{-\1/2}^{b}\psi_{-\1/2}^{c}\psi_{-\1/2}^{d}\Omega -\frac{i}{2}\sum_{a,b,c}\Gamma_{ab}^{e}[S_{-1}^{e}, \psi_{-\1/2}^{a}\psi_{-\1/2}^{b}   ] \Omega  \\=  - \frac{1}{12}\sum_{a,b,c,d}({\sum_{e}(\Gamma_{ab}^{e}\Gamma_{cd}^{e}+\Gamma_{da}^{e}\Gamma_{cb}^{e}+\Gamma_{db}^{e}\Gamma_{ac}^{e}) }  \psi_{-\1/2}^{a}\psi_{-\1/2}^{b}\psi_{-\1/2}^{c}\psi_{-\1/2}^{d}\Omega) \\   
   +  Ê \sum_{a,b,c,e}\Gamma_{ea}^{b}\Gamma_{ea}^{c}\psi_{-\3/2}^{c}\psi_{-\1/2}^{b} \Omega \   =  \    4g\omega    $
\end{proof}
  \begin{lemma}  \ (OPE and Lie bracket) 
   \begin{displaymath}   S^{a}(z)L(w) \sim  \frac{S^{a}(w)}{(z-w)^{2}}    \quad \mathrm{and}  \quad   [ L_{m}, S_{n}^{a} ] = -n  S_{m+n}^{a}     \end{displaymath}    \end{lemma} 
    \begin{proof}  $S_{n}^{a}.\omega  = 0$  for $n  \ge 3$ \   and: 
\begin{description}
\item[(a)]  $S_{0}^{a}.\omega = \frac{1}{4g} \sum_{b}S_{0}^{a}(S_{-1}^{b})^{2}\Omega =  \frac{1}{4g} \sum_{b}([S_{0}^{a},S_{-1}^{b}]S_{-1}^{b}\Omega + S_{-1}^{b} [S_{0}^{a},S_{-1}^{b}]\Omega) \\ = 
 \frac{i}{4g}\sum_{b,c}(\Gamma_{ab}^{c}+\Gamma_{ac}^{b} )S_{-1}^{c}S_{-1}^{b} \Omega = 0  $
\item[(b)]   $S_{2}^{a}.\omega = \frac{1}{4g} \sum_{b}([S_{2}^{a},S_{-1}^{b}]S_{-1}^{b}\Omega + S_{-1}^{b} [S_{2}^{a},S_{-1}^{b}]\Omega) \\ = \frac{i}{4g}\sum_{b,c}\Gamma_{ab}^{c} S_{1}^{c}S_{-1}^{b} \Omega = 
\frac{i}{4g}\sum_{b,c}\Gamma_{ab}^{c} \delta_{bc} \ell = 0 $  
\item[(c)]  $S_{1}^{a}.\omega = \frac{1}{4g} \sum_{b}([S_{1}^{a},S_{-1}^{b}]S_{-1}^{b}\Omega + S_{-1}^{b} [S_{1}^{a},S_{-1}^{b}]\Omega)\\  = \frac{i}{4g} ( 2 \ell + i \sum_{b,c}\Gamma_{ab}^{c} S_{0}^{c}S_{-1}^{b} \Omega  ) =\frac{2(\ell + g)}{4g} S_{-1}^{a}\Omega =  S_{-1}^{a}\Omega $  \ \ $(\star)$
\end{description}
   \end{proof} 
\begin{corollary} $(S^{a})$ generate a vertex operator (super)algebra with \\Ê $\omega =  \frac{1}{4g} \sum_{a}(S_{-1}^{a})^{2} \Omega $ as Virasoro vector.
\end{corollary}

\subsubsection{$\gg$-boson} 
\begin{definition} Let $  X^{a}(z)  =  \sum_{n \in \ZZZ}X_{n}^{a}z^{-n-1} $  the boson operators \\Êwith  $  [X_{m}^{a} , X_{n}^{b}] = [ X^{a} , X^{b}]_{m+n} + m\delta_{ab}\delta_{m+n}.\L  $ \end{definition}
\begin{corollary} The $\gg$-boson algebra $\widehat{\gg}_{+}$ generates a vertex operator \\  (super)algebra on $H = L(V_{0}, g )$, and also  on $H = L(V_{0}, \ell )$ for any 
$\ell \in \NNN$,   with $\omega =  \frac{1}{2(\ell+g)} \sum_{a}(X_{-1}^{a})^{2} \Omega $ as Virasoro vector; and:
\begin{displaymath}  X^{a}(z)X^{b}(w)   \sim   \frac{[X^{a}, X^{b}](w)}{(z-w)} + \frac{ g. \delta_{ab}}{(z-w)^{2}}     \end{displaymath}
\begin{displaymath}   X^{a}(z)L(w) \sim  \frac{X^{a}(w)}{(z-w)^{2}}    \quad \mathrm{and}  \quad   [ L_{m}, X_{n}^{a} ] = -n  X_{m+n}^{a}     \end{displaymath}   \end{corollary}
\begin{proof}    By the previous work on $(S^{a})$ and $(\star)$.   \end{proof}
\begin{lemma}  $c = 2 \Vert \omega \Vert ^{2} = \frac{\ell dim(\gg)}{\ell + g}$
\end{lemma}
\begin{proof}
$4(\ell+g)^{2} \Vert \omega \Vert ^{2} = \sum_{a,b}( (X_{-1}^{a})^{2} \Omega , (X_{-1}^{b})^{2} \Omega ) =  \sum_{a,b}(  \Omega , (X_{1}^{a})^{2}(X_{-1}^{b})^{2} \Omega )  \\ =  \sum_{a,b}(  \Omega , X_{1}^{a} X_{-1}^{b} [X_{1}^{a} ,  X_{-1}^{b} ] \Omega + X_{1}^{a} [X_{1}^{a} ,  X_{-1}^{b} ] X_{-1}^{b}\Omega ) \\ = Ê(\sum_{a,b,c}i\Gamma_{ab}^{c}( \Omega , X_{1}^{a}X_{0}^{c}X_{-1}^{b} \Omega )) + 2 \ell \sum_{a}(\Omega , X_{1}^{a}X_{-1}^{a} \Omega) \\  = (\sum_{a,b,c,d}(-1)\Gamma_{ab}^{c}\Gamma_{cb}^{d}(\Omega , X_{1}^{a}X_{-1}^{d}\Omega) + 2 \ell^{2} dim(\gg)) \\ =  (2g\ell dim(\gg) + 2 \ell^{2} dim(\gg)  ) = 2\ell dim(\gg)(\ell + g) $
\end{proof}
\begin{remark}  By vacuum axiom of vertex operator superalgebra, $X_{0}^{a}\Omega = 0$, then, the representation $H_{0} = V_{\lambda}$ of $\gg$ is necessary the trivial one $V_{0} $.  \\ 
At section  \ref{mod}, we  see that general $L(V_{\lambda} , \ell)$ admits the structure of vertex module over  $L(V_{0} , \ell)$. 
\end{remark}

\subsubsection{$\gg$-supersymmetry}  \label{super}
By lemma \ref{actiono}, the $\gg$-boson algebra $\widehat{\gg}_{+}$  acts on the $\gg$-fermion algebra $\widehat{\gg}_{+}$, then, we can build their semi-direct product:
\begin{definition}  \label{superg}
Let $\widehat{\gg} = \widehat{\gg}_{+} \ltimes \widehat{\gg}_{-} $ the $\gg$-supersymmetric algebra. \end{definition}
\begin{proposition}  \label{irredsl}
The unitary highest weight representations (irreducible) of $\widehat{\gg}$ are  $H = L(V_{\lambda}, \ell) \otimes \F_{NS}^{\gg}$ (see \cite{8b}).
\end{proposition}
\begin{proof}
Let $H$ be such a representation of $\widehat{\gg}$,  then, $\widehat{\gg}_{-}$ acts on, but it admits a unique irreducible representation: $\F_{NS}^{\gg}$, so $H =  M \otimes \F_{NS}^{\gg}$, with $M$ a multiplicity space.
Now, $\widehat{\gg}_{+}$ acts on $H$ and on $\F_{NS}^{\gg}$  (corollary  \ref{bosonaction} ), and the difference commutes with $\widehat{\gg}_{-}$; but 
$\widehat{\gg}_{-}$ acts irreducibly on $\F_{NS}^{\gg}$, so, the commutant of $\widehat{\gg}_{-}$ is $End(M) \otimes \CCC $ by Schur's lemma. So, $\widehat{\gg}_{+}$ acts on $M$, and this action is necessarily irreducible. Finally, by unitary highest weight context, $\exists \lambda$ such that  $M =  L(V_{\lambda}, \ell)$.
\end{proof} 
\begin{remark}  \label{B}
Using the previous notations, $\widehat{\gg}_{+}$  acts on $L(V_{\lambda}, \ell) \otimes \F_{NS}^{\gg}$  as  $B_{n}^{a} =  X_{n}^{a} + S_{n}^{a}$, bosons of level $d = \ell + g $. \end{remark}
\begin{corollary}
From $(\psi^{a}(z))$ and $(B^{a}(z))$, we  generate $S^{a}(z)$ and $X^{a} = B^{a} - S^{a}$ a vertex operator superalgebra on $H  = L(V_{0}, \ell) \otimes \F_{NS}^{\gg} $ with the Virasoro vector: 
\begin{displaymath}
\omega =  \1/2\sum_{a} \psi_{-\3/2}^{a}\psi_{-\1/2}^{a}\Omega +   \frac{1}{2(\ell+g)} \sum_{a}(X_{-1}^{a})^{2} \Omega \quad \mathrm{and:}
\end{displaymath}
\begin{displaymath}
c= 2 \Vert  \omega  \Vert ^{2} = \frac{dim(\gg)}{2} + \frac{\ell dim(\gg)}{\ell + g} =  \frac{3}{2}  \cdot \frac{\ell +\t1 g}{\ell + g} dim(\gg) 
\end{displaymath}  \end{corollary}
\begin{definition} (SuperVirasoro operator)  \\  \label{supertau}
Let $\tau_{1} = \sum_{a} \psi_{-\1/2}^{a}X_{-1}^{a}\Omega$,  $\tau_{2} =\frac{1}{3}  \sum_{a} \psi_{-\1/2}^{a}S_{-1}^{a}\Omega $ and $\tau = (\ell + g)^{-\1/2}(\tau_{1} +\tau_{2} ) $. \\
Let $G(z) = V(\tau , z) = \sum_{n \in \ZZZ}G_{n-\1/2}z^{-n-1} = \sum_{n \in \ZZZ + \1/2} G_{n}z^{-n-\3/2}  $
\end{definition}

\begin{proposition} (Supersymmetry boson-fermion) \label{susybf}  
\begin{displaymath}
  B^{a}(z)G(w) \sim   d^{\1/2} \frac{\psi^{a}(w)}{(z-w)^{2}}  \quad   \textrm{and}   \quad   \psi^{a}(z)G(w) \sim   d^{-\1/2} \frac{B^{a}(w)}{(z-w)} 
  \end{displaymath}
  \begin{displaymath}
   [ G_{m} ,  B^{a}_{n}] = -n d^{\1/2}  \psi^{a}_{m+n}   \quad \textrm{and}   \quad   [ G_{m} ,  \psi^{a}_{n}]_{+} =  d^{-\1/2}  B^{a}_{m+n}    
\end{displaymath}
\end{proposition}
\begin{proof}    
 $ \psi_{n+\1/2}^{a}\tau_{i} = 0$ for $n \ge 2 $ \  and:  
\begin{description}
\item[(a)]  $ \psi_{\1/2}^{a}\tau_{1} =   X_{-1}^{a} \Omega $ 
\item[(b)]  $ \psi_{\1/2}^{a}\tau_{2} =  \t1 (S_{-1}^{a} \Omega - \sum_{b}\psi_{-\1/2}^{b}\psi_{\1/2}^{a}S_{-1}^{b}\Omega ) =
 \t1 ( S_{-1}^{a} \Omega -i \sum_{b,c}\Gamma_{ab}^{c}\psi_{-\1/2}^{b}\psi_{-\1/2}^{c}\Omega) \\  = S_{-1}^{a} \Omega$ 
\item[(c)]   $\psi_{\3/2}^{a}\tau_{1} = \psi_{\3/2}^{a}\tau_{2}  = 0$. \end{description} 
$ S_{n}^{a}\tau_{i},  X_{n}^{a}\tau_{i} = 0$ for $n \ge 2 $ \  and: 
\begin{description}
\item[(a)]  $S_{0}^{a}\tau_{1}  =  \sum_{b}S_{0}^{a}\psi_{-\1/2 }^{b}X_{-1}^{b}\Omega = i\sum_{b,c} \Gamma_{ab}^{c}\psi_{-\1/2 }^{c}X_{-1}^{b}\Omega $
\item[(b)]  $ S_{0}^{a}\tau_{2}  =  \frac{1}{3}  \sum_{b}S_{0}^{a} \psi_{-\1/2}^{b}S_{-1}^{b}\Omega =   \frac{1}{3} ( i\sum_{b,c}\Gamma_{ab}^{c} \psi_{-\1/2}^{c}S_{-1}^{b}\Omega +  \sum_{b} \psi_{-\1/2}^{b}S_{0}^{a}S_{-1}^{b}\Omega) \\ = \frac{1}{3} ( i\sum_{b,c}\Gamma_{ab}^{c} \psi_{-\1/2}^{c}S_{-1}^{b}\Omega + i\sum_{b,c}\Gamma_{ab}^{c} \psi_{-\1/2}^{b}S_{-1}^{c}\Omega) \\=  \frac{i}{3}  \sum_{b,c}(\Gamma_{ab}^{c}+ \Gamma_{ac}^{b}) \psi_{-\1/2}^{c}S_{-1}^{b}\Omega  = 0 $
\item[(c)]   $X_{0}^{a}\tau_{1}  =  \sum_{b}\psi_{-\1/2 }^{b}X_{0}^{a}X_{-1}^{b} \Omega =   i \sum_{b,c}\Gamma_{ab}^{c}\psi_{-\1/2 }^{b}X_{-1}^{c}\Omega = -S_{0}^{a}\tau_{1}$
\item[(d)]  $X_{0}^{a}\tau_{2}  = X_{1}^{a}\tau_{2} = S_{1}^{a}\tau_{1} = 0 $
\item[(e)]  $X_{1}^{a}\tau_{1} =  \ell \psi_{-\1/2}^{a} \Omega$
\item[(f)]  $ S_{1}^{a}\tau_{2}  =  \frac{1}{3}  \sum_{b}S_{1}^{a} \psi_{-\1/2}^{b}S_{-1}^{b}\Omega =   \frac{1}{3} ( i\sum_{b,c}\Gamma_{ab}^{c} \psi_{\1/2}^{c}S_{-1}^{b}\Omega +  \sum_{b} \psi_{-\1/2}^{b}S_{1}^{a}S_{-1}^{b}\Omega) \\  =  \frac{1}{3} (\sum_{b,c,d} \Gamma_{ab}^{c}\Gamma_{bc}^{d}\psi_{-\1/2}^{d} \Omega + g\psi_{-\1/2}^{a} \Omega) =  g\psi_{-\1/2}^{a} \Omega$
\end{description}
\end{proof}

\begin{remark}   $G_{m}^{\star} =  G_{-m}$  (as lemma \ref{etoile})
\end{remark}

\begin{lemma}  \ (OPE and Lie bracket)
\begin{displaymath}
  L(z)G(w) \sim    \frac{G'(w)}{(z-w)} +  \frac{\3/2 G(w)}{(z-w)^{2}}  \quad   \textrm{and}   \quad  [G_{m },  L_{n} ] = (m-\1/2n)G_{m+n}
\end{displaymath}  \end{lemma}
\begin{proof} $L(n)\tau =  L_{n-1} \tau = 0  $ for $n \ge 3 $ \  and:
\begin{description}
\item[(a)]  $L_{-1} \tau =  R(G') $   (see $L_{-1}$ axioms and definition \ref{R}) 
\item[(b)]   $L_{0} \tau = \3/2 R(G) $  (see $L_{0}$ axioms)
\item[(c)]   $L_{1} (\tau_{1}+\tau_{2} ) = \sum_{a}L_{1} \psi_{-\1/2}^{a} (X_{-1}^{a} + \t1 S_{-1}^{a}  ) \Omega = \sum_{a} \psi_{-\1/2}^{a}L_{1} (X_{-1}^{a} + \t1 S_{-1}^{a}  ) \Omega \\  = \sum_{a} \psi_{-\1/2}^{a} (X_{0}^{a} + \t1 S_{0}^{a}  ) \Omega = 0 $
\end{description} \end{proof}

\begin{remark}
$[[A ,B ]_{+}, C  ] = [  A , [B,  C]_{+}  ] +[  B , [A,  C]_{+}  ] \\ 
\begin{array}{c} \end{array}  \hspace{5cm} =   [  A , [B,  C]  ]_{+} +[  B , [A,  C]  ] _{+} $ 
\end{remark}
\begin{lemma}  \ (OPE and Lie bracket)
\begin{displaymath}
  G(z) G(w)  \sim     \frac{\frac{2}{3} c}{(z-w)^{3}} +   \frac{2  L(w) }{(z-w)}    \quad   \textrm{and}   \quad  \lbrack G_{m},G_{n}\rbrack _{+} = 2L_{m+n} + \frac{c}{3}(m^{2}- \frac{1}{4})\delta_{m+n}  \end{displaymath}  \end{lemma}
\begin{proof}By supersymmetry:  
\begin{description}
\item[(a)]   $[[G_{m } ,G_{n } ]_{+}, B^{a}_{r}  ]  =  -2r B^{a}_{m+n+r} = [2L_{m+n} , B^{a}_{r}  ]$
\item[(b)]  $[[G_{m } ,G_{n } ]_{+}, \psi^{a}_{r}  ]  =  -2(r+\1/2(m+n)) \psi^{a}_{m+n+r} = [2L_{m+n} , \psi^{a}_{r}  ]$
\end{description}
Then, $[[G_{m } ,G_{n } ]_{+} - 2L_{m+n}, B^{a}_{r}  ] = [[G_{m } ,G_{n } ]_{+} -2L_{m+n}, \psi^{a}_{r}  ]  = 0   $. \\
Now, $(B^{a}_{r}  )$, $ ( \psi^{a}_{r} )$  act irreducibly on $H$, so by Schur's lemma:   
\begin{displaymath}  [G_{m } ,G_{n } ]_{+} - 2L_{m+n} =  k_{m,n} I  \end{displaymath}
Now, among the $G_{n }\tau $,  $G_{\3/2 }\tau $ is the only to give a constant term and: \\
 $G_{\3/2 }\tau = (\ell + g)^{-1} \sum_{a} G_{\3/2 }\psi^{a}_{-\1/2} (X_{-1}^{a} + \t1 S_{-1}^{a} )\Omega \\  =  (\ell + g)^{-1} \sum_{a} (X_{1}^{a} +  S_{1}^{a} ) (X_{-1}^{a} + \t1 S_{-1}^{a} )\Omega
\\ =  (\ell + g)^{-1} dim(\gg) (\ell +\t1 g )\Omega = \frac{2}{3} c \Omega $. \\ 
Finally, by  formulas \ref{formula} and  \ref{lie}, \    $k_{m,n} =  \frac{c}{3} (m^{2}-\frac{1}{4}) \delta_{m+n} $. \end{proof}

\begin{summary}
\begin{displaymath} 
\left\{    \begin{array}{l}   L(z) L(w)  \sim     \frac{(c/2)}{(z-w)^{4}} +   \frac{2  L(w) }{(z-w)^{2}} +   \frac{ L'(w) }{(z-w)} 
       \\  L(z)G(w) \sim    \frac{G'(w)}{(z-w)} +  \frac{\3/2 G(w)}{(z-w)^{2}}    \\     G(z) G(w)  \sim     \frac{\frac{2}{3} c}{(z-w)^{3}} +   \frac{2  L(w) }{(z-w)}   
      \end{array}  \right.
\end{displaymath}
and: 
\begin{displaymath} 
\left\{     \begin{array}{l}  \lbrack L_{m},L_{n} \rbrack  = (m-n)L_{m+n} +\frac{c}{12}(m^{3} - m) \delta_{m+n}    
 \\   \lbrack G_{m},L_{n}\rbrack  = (m-\frac{n}{2})G_{m+n}      
   \\      \lbrack G_{m},G_{n}\rbrack _{+} = 2L_{m+n} + \frac{c}{3}(m^{2}- \frac{1}{4})\delta_{m+n}   \end{array}  \right.  
 \end{displaymath}
 \begin{displaymath}L_{n}^{\star} =  L_{-n},  \  G_{m}^{\star} =  G_{-m},  \ \  \textrm{and}  \   c =  \frac{3}{2}  \cdot \frac{\ell +\t1 g}{\ell + g} dim(\gg)  \end{displaymath}
the SuperVirasoro algebra of sector (NS), or Neveu-Schwarz algebra     $ \Vir_{1/2}$.  \end{summary}

\begin{corollary}
$ \Vir_{\1/2}$ acts unitarily on $H = L(V_{0}, \ell)\otimes \F_{NS}^{\gg}$ and admits $L(c,0)$  as minimal submodule containing  Ê$\Omega$ (see definition \ref{Lhc} ).
\end{corollary}

\subsection{Vertex modules}  \label{mod}
\begin{remark} If $\ell = 0$, then $\lambda = 0$ and  $L(V_{0}, 0) = \CCC$ trivial,  and what we will show is ever proved by the previous section.
So, we  suppose $\ell \in \NNN^{\star}$ fixed.
\end{remark}
\subsubsection{Summary}
Let  $H = L(V_{0}, \ell)\otimes \F_{NS}^{\gg}$,  the vacuum representation of the $\gg$-supersymmetric algebra $\widehat{\gg}$,  with   $ \pi : \widehat{\gg}  \longrightarrow    End(H)$.  \\
We have construct the vertex operator superalgebra $(H, \Omega, \omega, V )$ with  \\ $V: H   \longrightarrow   (End H)[[z, z^{-1}]]$ the state-field correspondance map. \\
$\S = V(H)$ is generated by  $(V(\psi_{-\1/2}^{a}\Omega))_{a}$,  $(V(X_{-1}^{b}\Omega))_{b}$, and  $V(\L \Omega) $,  pairwise local,   with the operations, $(A,B) \mapsto A_{n}B $ and linear combinations. \\
We write   $V(\psi_{-\1/2}^{a}\Omega , z) = \sum_{n \in \ZZZ} \pi (\psi_{n+\1/2}^{a} ) z^{-n-1} $,  \\   $V(X_{-1}^{b}\Omega , z) = \sum_{n \in \ZZZ} \pi (X_{n}^{b} ) z^{-n-1} $ and $V(\L \Omega , z) = \pi ( \L) \  ( = \ell Id_{H} ) $.
 \subsubsection{Modules}
Let  $H^{\lambda} = L(V_{\lambda}, \ell)\otimes \F_{NS}^{\gg}$ a unitary highest weight representation of $\widehat{\gg}$ and $ \pi^{\lambda} : \widehat{\gg}  \longrightarrow    End(H^{\lambda})$
\begin{remark}
$H^{\lambda}$ is  itself the minimal subspace containing  $\Omega^{\lambda} $ and stable by the action of  $\widehat{\gg}$:  $\Omega^{\lambda} $ is  the cyclic vector of $H^{\lambda}$.  \\
On the vacuum representation,   $\Omega $ is called the vacuum vector.
\end{remark}
\begin{lemma}   \label{faithful}
$( \sum_{n \in \ZZZ} \pi^{\lambda} (\psi_{n+\1/2}^{a} ) z^{-n-1})_{a}$, $(\sum_{n \in \ZZZ} \pi^{\lambda} (X_{n}^{b} ) z^{-n-1})_{b}$ and  $\pi^{\lambda} ( \L) $ are pairwise local (definition \ref{local}).
\end{lemma}
\begin{proof}
Let $A$, $B  \in \widehat{\gg} [[z, z^{-1}]]$;  $\pi $ and $\pi^{\lambda}$ are faithful representations of  $\widehat{\gg}$. \\  Then, as formal power series,  with $N \in  \NNN$ and $\varepsilon \in \ZZZ_{2} $: \\
  $    (z-w)^{N}\pi(A(z))\pi(B(w))c,d) \\ = (-1)^{\varepsilon}   (z-w)^{N}(\pi(B(w))\pi(A(z))c,d)  \quad     \forall c, d \in H    \quad    \quad \textrm{if and only if} $  \\ 
  $  (z-w)^{N}(\pi^{\lambda}(A(z))\pi^{\lambda}(B(w))e,f)  \\ Ê= (-1)^{\varepsilon}   (z-w)^{N}(\pi^{\lambda}(B(w))\pi^{\lambda}(A(z))e,f)  \quad  \forall e, f \in H^{\lambda} $  \\ \\
We generate inductively an operator $D$ decomposing $H^{\lambda}$ into $\bigoplus H_{n}^{\lambda} $ by: \\
 $D \Omega^{\lambda} = 0$,   $D\psi_{-m}^{a} \xi = \psi_{-m}^{a} D \xi + m\psi_{-m}^{a} \xi $,  $DX_{-n}^{b} \xi = X_{-n}^{b} D \xi + nX_{-n}^{b} \xi$, $\xi \in H^{\lambda} $, clearly well defined; 
but, $\psi_{m}^{a}:  H_{p}^{\lambda}   \to  H_{p-m}^{\lambda}   $ and  $X_{n}^{b}:  H_{p}^{\lambda}   \to  H_{p-n}^{\lambda}   $,  so, by lemmas \ref{1}, \ref{2}, \ref{3},  the result follows.
\end{proof}

\begin{lemma} \label{exo}
$D = L_{0} -  \frac{c_{V_{\lambda}}}{2(\ell + g)}$ , \\  with $c_{V_{\lambda}}$ the Casimir number of $V_{\lambda}$ (see corollary \ref{cas})
 \end{lemma}
 \begin{proof}
$[L_{0}, \psi_{n}^{a}] = [D, \psi_{n}^{a}] $ and $[L_{0}, X_{n}^{a}] = [D, X_{n}^{a}] $, so, by irreducibility and Schur's lemma,  $ L_{0} - D \in \CCC Id_{H^{\lambda}}$.
Now, $D \Omega^{\lambda} = 0 $ and $L_{0} \Omega^{\lambda} = h \Omega^{\lambda} \neq 0$  in general.  Now, writing explicitly $L_{0}$  with formula \ref{formula}, we obtain: 

 \quad \quad  \quad \quad  $ 2(\ell + g) L_{0} \Omega^{\lambda} =  \sum_{a} (X_{0}^{a})^{2} \Omega^{\lambda} =  \C . \Omega^{\lambda} =  c_{V_{\lambda}} \Omega^{\lambda} $ 
  \end{proof} 
  \newpage 
  \begin{theorem}  \label{fari}
$ \Vir_{\1/2}$ acts unitarily on $H^{\lambda} = L(V_{\lambda}, \ell)\otimes \F_{NS}^{\gg}$ \\  and admits $L(c,h)$  as minimal submodule containing  Ê$\Omega^{\lambda}$, \\    with $c = \frac{3}{2}  \cdot \frac{\ell +\t1 g}{\ell + g} dim(\gg)$     and $h = \frac{c_{V_{\lambda}}}{2(\ell + g)}  $.  
 \end{theorem}
 
 \begin{proof}

  We generate $\S^{\lambda} $ from generators of previous lemma, with the operations   $(A, B) \mapsto A_{n}B $ (now available) and linear combinations. The formula \ref{formula} is independant of the choice between the faithful representations  $\pi $ and $\pi^{\lambda}$. So, we identify $\S$ and $\S^{\lambda}$, which gives the isomorphism  $i : \S \to    \S^{\lambda}$; we compose it with the state-field correspondence map $V : H \to \S$ to give:  
  \begin{displaymath}
   \begin{array}{cccc}
V^{\lambda}:  & H &  \longrightarrow  &  (End H^{\lambda})[[z,z^{-1}]]     \\
    &    a  &    \longmapsto  &  i(V(a))     \end{array}    \hspace{0,5cm}Ê(1)
 \end{displaymath}
 Then,  \quad   $\sum_{n \in \ZZZ} \pi^{\lambda} (\psi_{n+\1/2}^{a} ) z^{-n-1} =V^{\lambda}(\psi_{-\1/2}^{a}\Omega , z) $, \\  $\sum_{n \in \ZZZ} \pi^{\lambda} (X_{n}^{b} ) z^{-n-1}= V^{\lambda}(X_{-1}^{b}\Omega , z)$ \  and  \  $\pi^{\lambda} ( \L) = V^{\lambda}(\L \Omega , z) $
 
 Now,  $V(a)_{n}V(b)  = V(V(a,n)b)$  $\forall  a,  b \in H$, so, by construction:
  \begin{displaymath}  V^{\lambda}(a)_{n}V^{\lambda}(b)  = V^{\lambda}(V(a,n)b)  \hspace{0,5cm} (2)  \end{displaymath}
 Then, $V^{\lambda}(\omega , z) = \sum L_{n} z^{-n-2}$ ,  $V^{\lambda}(\tau , z  ) = \sum G_{m- \1/2} z^{-m-1} $,   $L_{n}^{\star} =  L_{-n}$   and   $G_{m}^{\star} =  G_{-m}$, with $(L_{n})$,  $(G_{m})$ verifying superVirasoro  relations. $  \hspace{0,5cm}  (3)$
 \end{proof}

\begin{remark}     $[L_{m}, \psi_{n}^{a}] =  -(n+\1/2 m ) \psi_{m+n}^{a}$  and  $[L_{m}, X_{n}^{a}] =  -n X_{m+n}^{a} $, so:  
\begin{displaymath}  \left\{   \begin{array}{l}   \lbrack L_{-1} , V^{\lambda}(a,z) \rbrack  = (V^{\lambda})'(a,z) \\    \lbrack L_{0} , V^{\lambda}(a,z) \rbrack = z.(V^{\lambda})'(a,z) + r V^{\lambda}(a,z)  \quad  (a \in H_{r})  \end{array} \right.   \hspace{0,5cm} (4) \end{displaymath}  \end{remark}

\begin{remark}
$V^{\lambda}( \Omega , z) = Id_{H^{\lambda}}$ because $\pi $ and $\pi^{\lambda}$ are at same level $\ell$.  $\hspace{0,5cm} (5)$

 \end{remark}
 
 \begin{definition} \label{moduledef} By $(1) ... (5)$, $(H^{\lambda} , V^{\lambda} )$ is called a \textbf{vertex module} of $(H , V , \Omega, \omega) $.
\end{definition}
We now apply the theorem \ref{fari} to GKO  construction with $\gg = \sl_{2}$.

  \end{document}